\documentclass[a4paper,11pt]{article}
\usepackage{amsmath,amsthm}
\usepackage{authblk}
\usepackage[style=numeric-comp,maxbibnames=99,bibencoding=utf8,firstinits=true]{biblatex}
\usepackage{booktabs}
\usepackage[makeroom]{cancel}
\usepackage{enumitem}
\usepackage[hmargin={26mm,26mm},vmargin={30mm,35mm}]{geometry}
\usepackage{nicefrac}
\usepackage[colorlinks,allcolors={blue}]{hyperref}
\usepackage[compact,small]{titlesec}
\usepackage{tikz-cd}

\usepackage{newtxtext}
\usepackage{newtxmath}

\bibliography{cohomology}

\newcommand{\email}[1]{\href{mailto:#1}{#1}}

\numberwithin{equation}{section}

\newtheorem{theorem}{Theorem}
\newtheorem{proposition}[theorem]{Proposition}
\newtheorem{lemma}[theorem]{Lemma}
\newtheorem{corollary}[theorem]{Corollary}
\theoremstyle{remark}
\newtheorem{remark}[theorem]{Remark}
\theoremstyle{definition}

\newcommand{\st}{\,:\,}
\newcommand{\Real}{\mathbb{R}}

\DeclareRobustCommand{\bvec}[1]{\boldsymbol{#1}}
\newcommand{\uvec}[1]{\underline{\bvec{#1}}}
\newcommand{\cvec}[1]{\bvec{\mathcal{#1}}}

\DeclareMathOperator{\GRAD}{\bf grad}
\DeclareMathOperator{\CURL}{\bf curl}
\DeclareMathOperator{\DIV}{div}
\DeclareMathOperator{\ROT}{rot}
\DeclareMathOperator{\VROT}{\bf rot}

\newcommand{\compl}{{\rm c}}
 
\newcommand{\Hcurl}[1]{\bvec{H}(\CURL;#1)}

\newcommand{\Hdiv}[1]{\bvec{H}(\DIV;#1)}

\newcommand{\Xgrad}[2]{\underline{X}_{\GRAD,#2}^{#1}}
\newcommand{\Xcurl}[2]{\underline{\bvec{X}}_{\CURL,#2}^{#1}}
\newcommand{\Xdiv}[2]{\underline{\bvec{X}}_{\DIV,#2}^{#1}}
\newcommand{\Xbullet}[2]{\underline{X}_{\bullet,#2}^{#1}}

\newcommand{\XgradZR}[2]{\widetilde{\underline{X}}_{\GRAD,#2}^{#1}}
\newcommand{\XcurlZR}[2]{\widetilde{\underline{\bvec{X}}}_{\CURL,#2}^{#1}}
\newcommand{\XdivZR}[2]{\widetilde{\underline{\bvec{X}}}_{\DIV,#2}^{#1}}
\newcommand{\XbulletZR}[2]{\widetilde{\underline{X}}_{\bullet,#2}^{#1}}

\newcommand{\Xg}[1]{X_{#1}}

\newcommand{\dg}[1]{{\rm d}_{#1}}

\newcommand{\tXg}[1]{\widetilde{X}_{#1}}

\newcommand{\tdg}[1]{{\rm \widetilde{d}}_{#1}}

\newcommand{\Eg}[1]{E_{#1}}
\newcommand{\Rg}[1]{R_{#1}}
\newcommand{\Egrad}[1]{\underline{E}_{\GRAD,#1}}
\newcommand{\EPoly}[1]{E_{\mathcal{P},#1}^{k-1}}
\newcommand{\Rgrad}[1]{\underline{R}_{\GRAD,#1}}
\newcommand{\Ecurl}[1]{\uvec{E}_{\CURL,#1}}
\newcommand{\ERoly}[1]{\bvec{E}_{\cvec{R},#1}^{k-1}}
\newcommand{\Rcurl}[1]{\uvec{R}_{\CURL,#1}}
\newcommand{\Ediv}[1]{\uvec{E}_{\DIV,#1}}
\newcommand{\EGoly}[1]{\bvec{E}_{\cvec{G},#1}^{k-1}}
\newcommand{\Rdiv}[1]{\uvec{R}_{\DIV,#1}}

\newcommand{\Rbullet}[1]{\underline{R}_{\bullet,#1}}

\newcommand{\dRmgrad}[1]{\kappa_{\GRAD}}
\newcommand{\dRmcurl}[1]{\kappa_{\CURL}}
\newcommand{\dRmdiv}[1]{\kappa_{\DIV}}
\newcommand{\dRmpoly}[1]{\kappa_{\Poly{}}}

\newcommand{\Id}{{\rm Id}}

\newcommand{\sfP}{{\mathsf{P}}}

\newcommand{\Igrad}[2]{\underline{I}_{\GRAD,#2}^{#1}}

\newcommand{\lproj}[2]{\pi_{\mathcal{P},#2}^{#1}}

\newcommand{\Rproj}[2]{\bvec{\pi}_{\cvec{R},#2}^{#1}}
\newcommand{\Rcproj}[2]{\bvec{\pi}_{\cvec{R},#2}^{\compl,#1}}

\newcommand{\Gproj}[2]{\bvec{\pi}_{\cvec{G},#2}^{#1}}
\newcommand{\Gcproj}[2]{\bvec{\pi}_{\cvec{G},#2}^{\compl,#1}}

\newcommand{\uGT}[1]{\uvec{G}_T^{#1}}
\newcommand{\uGTp}[1]{\uvec{G}_{T'}^{#1}}
\newcommand{\uGF}[1]{\uvec{G}_F^{#1}}

\newcommand{\uCT}[1]{\uvec{C}_T^{#1}}

\newcommand{\uGh}[1]{\uvec{G}_h^{#1}}
\newcommand{\uCh}[1]{\uvec{C}_h^{#1}}
\newcommand{\Dh}[1]{D_h^{#1}}

\newcommand{\GE}[1]{G_E^{#1}}
\newcommand{\cGF}[1]{\boldsymbol{\mathsf{G}}_F^{#1}}
\newcommand{\cGT}[1]{\boldsymbol{\mathsf{G}}_T^{#1}}

\newcommand{\CF}[1]{C_F^{#1}}
\newcommand{\cCT}[1]{\boldsymbol{\mathsf{C}}_T^{#1}} %% \accentset{\bullet}{\bvec{C}}_T^k}

\newcommand{\DT}[1]{D_T^{#1}}

\newcommand{\trE}{\gamma_E^{k+1}}
\newcommand{\trF}{\gamma_F^{k+1}}
\newcommand{\trFt}[1]{\bvec{\gamma}_{{\rm t},F}^{#1}}

\newcommand{\elements}[1]{\mathcal{T}_{#1}}
\newcommand{\faces}[1]{\mathcal{F}_{#1}}
\newcommand{\edges}[1]{\mathcal{E}_{#1}}
\newcommand{\vertices}[1]{\mathcal{V}_{#1}}

\newcommand{\FT}{\faces{T}}
\newcommand{\ET}{\edges{T}}
\newcommand{\EF}{\edges{F}}
\newcommand{\VT}{\vertices{T}}
\newcommand{\VF}{\vertices{F}}
\newcommand{\VE}{\vertices{E}}

\newcommand{\normal}{\bvec{n}}
\newcommand{\tangent}{\bvec{t}}

\newcommand{\Poly}[1]{\mathcal{P}^{#1}}

\newcommand{\vPoly}[1]{\cvec{P}^{#1}}
\newcommand{\Roly}[1]{\cvec{R}^{#1}}
\newcommand{\Goly}[1]{\cvec{G}^{#1}}
\newcommand{\cRoly}[1]{\cvec{R}^{\compl,#1}}
\newcommand{\cGoly}[1]{\cvec{G}^{\compl,#1}}

\DeclareMathOperator{\Ker}{Ker}
\DeclareMathOperator{\Image}{Im}

\newcommand{\Mh}{\mathcal{M}_{h}}
\newcommand{\Th}{\elements{h}}
\newcommand{\Fh}{\faces{h}}
\newcommand{\Eh}{\edges{h}}
\newcommand{\Vh}{\vertices{h}}

\newcommand{\Pcurl}[1]{\bvec{P}_{\CURL,T}^{#1}}
\newcommand{\Pdiv}[1]{\bvec{P}_{\DIV,T}^{#1}}

\newcommand{\DDR}[1]{{\rm DDR}(\ensuremath{#1})}

\newcommand{\Ebullet}[1]{\underline{E}_{\bullet,#1}}
\newcommand{\Pibullet}[1]{\underline{\Pi}_{\bullet,#1}}

\newcommand{\Dgrad}{\underline{\mathcal D}_{\GRAD}^k}
\newcommand{\Dcurl}{\underline{\mathcal D}_{\CURL}^k}
\newcommand{\Ddiv}{\underline{\mathcal D}_{\DIV}^k}

\newcommand{\coH}[1]{\mathcal{H}^{#1}}

\usepackage{pgfplots,pgfplotstable}
\usepackage{tikz-cd}

\graphicspath{{figures/}}

\begin{document}

\title{Cohomology of the discrete de Rham complex on domains of general topology}

\author[1]{Daniele A. Di Pietro}
\author[2]{J\'er\^ome Droniou}
\author[3]{Silvano Pitassi}
\affil[1]{IMAG, Univ Montpellier, CNRS, Montpellier, France, \email{daniele.di-pietro@umontpellier.fr}}
\affil[2]{School of Mathematics, Monash University, Melbourne, Australia, \email{jerome.droniou@monash.edu}}
\affil[3]{DPIA, University of Udine, Udine, Italy, \email{pitassi.silvano@spes.uniud.it}}
\maketitle

\begin{abstract}
  In this work we prove that, for a general polyhedral domain of $\Real^3$, the cohomology spaces of the discrete de Rham complex of [Di Pietro and Droniou, An arbitrary-order discrete de Rham complex on polyhedral meshes: Exactness, Poincaré inequalities, and consistency, Found. Comput. Math. 23, pp. 85--164, 2023, DOI: \href{https://dx.doi.org/10.1007/s10208-021-09542-8}{10.1007/s10208-021-09542-8}] are isomorphic to those of the continuous de Rham complex.
  This is, to the best of our knowledge, the first result of this kind for a fully computable arbitrary-order complex built from a general polyhedral mesh.
\end{abstract}

%% \tableofcontents

%------------------------------------------------------------------------------%

\section{Introduction}

The well-posedness of relevant classes of partial differential equations hinges on subtle analytical and homological properties that underpin \emph{Hilbert complexes}.
A Hilbert complex is a sequence of Hilbert spaces $X_i$ connected by closed densely defined linear operators $\dg{i}:X_i\to X_{i+1}$ such that the \emph{complex property} holds, i.e., the range of $\dg{i}$ is contained in the kernel of $\dg{i+1}$.
The best-known example is the de Rham complex which, for a connected domain $\Omega$ of $\Real^3$, reads
\begin{equation}\label{eq:de-rham}
  \begin{tikzcd}
    \Real\arrow{r}{\mathfrak I} & H^1(\Omega)
    \arrow{r}{\GRAD} & \Hcurl{\Omega}
    \arrow{r}{\CURL} & \Hdiv{\Omega}
    \arrow{r}{\DIV} & L^2(\Omega)
    \arrow{r}{0} & 0,
  \end{tikzcd}
\end{equation}
where $\mathfrak I$ identifies real numbers with constant functions,
$H^1(\Omega)$ is spanned by scalar-valued functions that are square-integrable over $\Omega$ along with their gradient,
while $\Hcurl{\Omega}$ and $\Hdiv{\Omega}$ are spanned by vector-valued functions that are square-integrable over $\Omega$ along with their curl and divergence, respectively.
The complex property corresponds, in this case,  to the classical relations $\GRAD\mathfrak{I} = \bvec{0}$, $\CURL\GRAD = \bvec{0}$, and $\DIV\CURL = 0$.
We notice that we moreover have $\Image \DIV=L^2(\Omega)$. Depending
on the topology of $\Omega$, the previous properties can become stronger.
Specifically, if $\Omega$ is not crossed by any tunnel (i.e., its first Betti number $b_1$ is zero), then $\Image\GRAD = \Ker\CURL$.
Similarly, if $\Omega$ does not enclose any void (i.e., its second Betti number $b_2$ is zero), then $\Image\CURL = \Ker\DIV$.
When both these properties hold, the complex is said to be \emph{exact}.
For domains with more complicated topologies, the defect of exactness is reflected by the fact that certain among the following \emph{cohomology spaces} may be non-trivial:
\begin{equation}\label{eq:cohomology.spaces}
  \begin{alignedat}{4}
    \coH{0}&\coloneq\Ker\GRAD/\Image\mathfrak{I},&\qquad
    \coH{1} &\coloneq\Ker\CURL/\Image\GRAD,
    \\
    \coH{2}&\coloneq\Ker\DIV/\Image\CURL,&\qquad
    \coH{3}&\coloneq L^2(\Omega)/\Image\DIV.
  \end{alignedat}
\end{equation}
It holds $\coH{0}=\{0\}$ (since $\Omega$ is connected) and $\coH{3}=\{0\}$ (since $\Image\DIV=L^2(\Omega)$ as mentioned above).
However, $\coH{1}$ or $\coH{2}$ are non-trivial if, respectively, $\Omega$ is crossed by tunnel(s) or encloses void(s).
As a matter of fact, the de Rham theorem states that the above cohomology spaces, with the exception of $\coH{0}$, are isomorphic to the simplicial cohomology spaces through the \emph{de Rham maps} \cite{Dodziuk:76}.
Notice, in passing, that the isomorphism could be extended to $\coH{0}$ with the usual choice of replacing $\Real$ with the trivial space at the beginning of the sequence, but this would result in the impossibility for the complex to be exact.
An important consequence of the de Rham theorem is that each $\coH{i}$, $i\in\{1,2,3\}$, has dimension equal to $b_i$.

The \emph{compatible} numerical approximation of problems whose stability hinges on Hilbert complexes is based on discrete versions of the relevant complex with cohomology spaces isomorphic to the continuous ones.
The first Finite Element approximations of vector-valued spaces in the de Rham sequence date from the late 70s \cite{Raviart.Thomas:77,Nedelec:80}.
Discrete $\Hdiv{\Omega}$- and $L^2(\Omega)$-conforming Finite Element spaces mimicking the exactness property of the rightmost portion of the de Rham complex have been used since the early 80s to prove the stability of mixed formulations of scalar diffusion problems \cite{Douglas.Roberts:82,Douglas.Roberts:85}.
The first use of the full de Rham complex, on the other hand, was made a few years later to devise stable approximations of problems in computational electromagnetism \cite{Bossavit:88}.
In recent years, the study of compatible Finite Elements has gravitated towards generalisations based on the formalism of exterior calculus (see, e.g., \cite{Arnold:18} and references therein).

While very powerful from certain points of view, the Finite Element paradigm is typically limited to meshes composed of elements with a limited number of shapes (usually, tetrahedra or hexahedra), which makes meshing complicated geometries or local mesh refinement more challenging.
To circumvent this limitation, \emph{polytopal paradigms} have emerged enabling the support of much more general meshes, including, e.g., non-matching interfaces and polyhedral elements.
The high-level approach of polytopal paradigms can also lead, even on standard (e.g., hexahedral) meshes, to a lower number of unknowns compared to Finite Elements; see, e.g., \cite[Table 2]{Di-Pietro.Droniou:21*2}.
A first example of low-order polyhedral discrete Hilbert complex has been obtained using the Mimetic Finite Difference method (see, e.g., \cite{Beirao-da-Veiga.Lipnikov.ea:14} and references therein).
Related approach are the Discrete Geometric Approach \cite{Codecasa.Specogna.ea:09} and Compatible Discrete Operators \cite{Bonelle.Ern:14,Bonelle.Di-Pietro.ea:15}, that hinge on the notions of dual mesh and discrete Hodge operators.

More recent developments have focused on the extension to high order of accuracy.
We can cite, e.g., the Virtual Element Method, which has been used in \cite{Beirao-da-Veiga.Brezzi.ea:16,Beirao-da-Veiga.Brezzi.ea:18*2} to devise arbitrary-order discrete Hilbert complexes on general polyhedral meshes.

Recent works \cite{Di-Pietro.Droniou.ea:20,Di-Pietro.Droniou:21*2} have introduced a fully discrete approach to the design and analysis of arbitrary-order discrete de Rham (DDR) complexes on general polyhedral meshes; see \cite{Di-Pietro.Droniou:21*1,Di-Pietro.Droniou:21*3,Hanot:21,Di-Pietro.Droniou:22,Di-Pietro.Droniou:22*1,Botti.Di-Pietro.ea:22} for applications and further developments, as well as \cite{Beirao-da-Veiga.Dassi.ea:22} for an in-depth study of the links with the Virtual Element Method.
The basic idea of the DDR approach is to replace both spaces and operators with discrete counterparts.
The discrete operators are constructed, through discrete integration by parts formulas, in order to fulfill suitable polynomial consistency properties.
Assuming $\Omega$ polyhedral and following the notations of \cite{Di-Pietro.Droniou.ea:20,Di-Pietro.Droniou:21*2}, the DDR complex corresponding to a polynomial degree $k\ge 0$ reads
\begin{equation}\label{eq:ddr.complex}
  \DDR{k}\coloneq
  \begin{tikzcd}
    \Real\arrow{r}{\Igrad{k}{h}} &[1.5em] \Xgrad{k}{h}
    \arrow{r}{\uGh{k}} & \Xcurl{k}{h}
    \arrow{r}{\uCh{k}} & \Xdiv{k}{h} 
    \arrow{r}{\Dh{k}}  & \Poly{k}(\Th) 
    \arrow{r} & 0.
  \end{tikzcd}
\end{equation}
Precise definitions of the spaces and operators above are provided in Section \ref{sec:ddr.complex}.
The cohomology spaces of this complex are
\begin{equation}\label{eq:cohomology.spaces:discrete}
  \begin{alignedat}{4}
  \coH{0,(k)}&\coloneq\Ker\uGh{k}/\Image\Igrad{k}{h},&\quad
  \coH{1,(k)} &\coloneq\Ker\uCh{k}/\Image\uGh{k},
  \\
  \coH{2,(k)}&\coloneq\Ker\Dh{k}/\Image\uCh{k},&\quad
  \coH{3,(k)}&\coloneq \Poly{k}(\Th)/\Image\Dh{k}.
  \end{alignedat}
\end{equation} 
It has been proved in \cite[Theorems 1 and 2]{Di-Pietro.Droniou:21*2} that these cohomology spaces are trivial for domains with trivial topology (i.e., such that $b_0 = 1$ and $b_1 = b_2 = b_3 = 0$), so that the $\DDR{k}$ complex is exact in this case.
The purpose of this paper is to study these spaces for domains with non-trivial topologies by proving the following theorem:

\begin{theorem}[Cohomology of the $\DDR{k}$ complex]\label{thm:ddr:cohomology}
  For any $k\ge 0$, the cohomology spaces defined by \eqref{eq:cohomology.spaces:discrete} are isomorphic to the de Rham cohomology spaces \eqref{eq:cohomology.spaces}.
\end{theorem}

In the context of Finite Element methods on standard meshes (interpreted in the exterior calculus framework), this isomorphism is proved in \cite[Section 5.5]{Arnold.Falk.ea:06}. Some polytopal methods are covered by \cite[Proposition 5.16]{Christiansen.Munthe-Kaas.ea:11}, but in a theoretical framework that does not describe \emph{computable} polytopal methods, and does not immediately give a practical way to compute generators of the discrete cohomology groups. To the best of our knowledge, Theorem \ref{thm:ddr:cohomology} is the first of its kind for an arbitrary-order fully computable polytopal complex; moreover, the technique of proof we develop provides an explicit and cheap way to describe the cohomology spaces (see Remark \ref{rem:generators} at the end of the paper).

In the case of a topologically trivial domain, Theorem \ref{thm:ddr:cohomology} also establishes the exactness of the global DDR complex in a more straightforward way than the previous proofs in \cite{Di-Pietro.Droniou:21*1,Di-Pietro.Droniou:21*2}, using only the local exactness in each element.
In a nutshell, the idea of the proof of Theorem \ref{thm:ddr:cohomology} consists in showing that, for any $k\ge 1$, the cohomology spaces \eqref{eq:cohomology.spaces:discrete} are isomorphic to $\coH{0,(0)}$, $\coH{1,(0)}$, $\coH{2,(0)}$, and $\coH{3,(0)}$, and, connecting through the de Rham map the DDR(0) complex to the CW complex defined by the mesh, that these spaces are in turn isomorphic to the de Rham cohomology spaces \eqref{eq:cohomology.spaces}.
This is done using the abstract framework originally introduced in \cite[Section 2]{Di-Pietro.Droniou:22} in the context of serendipity DDR methods.
In passing, as a consequence of Theorem \ref{thm:ddr:cohomology} and of the discussion in \cite[Section 6.6]{Di-Pietro.Droniou:22}, we immediately have also the following result:

\begin{corollary}[Cohomology of the serendipity DDR complex]
  The cohomology spaces of the serendipity DDR complex presented in \cite[Section 5]{Di-Pietro.Droniou:22} are isomorphic to the de Rham cohomology spaces \eqref{eq:cohomology.spaces}.
\end{corollary}

The rest of this work is organised as follows.
In Section \ref{sec:setting} we introduce the general setting (mesh, polynomial spaces, etc.).
In Section \ref{sec:ddr.complex} we briefly recall the DDR complex of \cite{Di-Pietro.Droniou:21*2}.
Section \ref{sec:cohomology} contains the proof of Theorem \ref{thm:ddr:cohomology},
and Section \ref{sec:conclusion} provides a brief conclusion and perspectives to this work.

%------------------------------------------------------------------------------%

\section{Setting}\label{sec:setting}

\subsection{Domain and mesh}

Denote by $\Omega\subset\Real^3$ a connected polyhedral domain.
We consider a polyhedral mesh $\Mh\coloneq\Th\cup\Fh\cup\Eh\cup\Vh$, where $\Th$ gathers the elements, $\Fh$ the faces, $\Eh$ the edges, and $\Vh$ the vertices. Both elements and faces are assumed to be topologically trivial.
The notations and assumptions are as in \cite{Di-Pietro.Droniou:21*2}. In particular, for each face $F\in\Fh$, we fix a unit normal $\normal_F$ to $F$ and, for each edge $E\in\Eh$, a unit tangent $\tangent_E$. Given $T\in\Th$, $\FT$ gathers the faces on the boundary $\partial T$ of $T$ and $\ET$ the edges in $\partial T$. For $F\in\FT$, $\omega_{TF}\in\{-1,+1\}$ is such that $\omega_{TF}\normal_F$ is the outer normal on $F$ to $T$.

Each face $F\in\Fh$ is oriented counter-clockwise with respect to $\normal_F$ and, for $E\in\EF$ with $\EF$ set of edges of $F$, we let $\omega_{FE}\in\{-1,+1\}$ be such that $\omega_{FE}=+1$ if $\tangent_E$ points along the boundary $\partial F$ of $F$ in the clockwise sense, and $\omega_{FE}=-1$ otherwise; we also denote by $\normal_{FE}$ the unit normal vector to $E$, in the plane spanned by $F$, such that $(\tangent_E,\normal_{FE},\normal_F)$ is a right-handed system of coordinate; it can be checked that $\omega_{FE}\normal_{FE}$ points outside $F$.
For all $V\in\Vh$, $\bvec{x}_V\in\Real^3$ denotes the coordinate vector of $V$.
For any mesh face $F\in\Fh$, we denote by $\GRAD_F$ and $\DIV_F$ the tangent gradient and divergence operators acting on smooth enough functions.
Moreover, for any $r:F\to\Real$ and $\bvec{z}:F\to\Real^2$ smooth enough, we let
$\VROT_F r\coloneq (\GRAD_F r)^\perp$ and
$\ROT_F\bvec{z}=\DIV_F(\bvec{z}^\perp)$,
with $\perp$ denoting the rotation of angle $-\frac\pi2$ in the oriented tangent space to $F$.

The mesh we consider is such that $(\Th,\Fh)$ belongs to a regular sequence as per \cite[Definition 1.9]{Di-Pietro.Droniou:20}.
This assumption ensures the existence, for each $\sfP\in\Th\cup\Fh\cup\Eh$, of a point $\bvec{x}_{\sfP}\in\sfP$ such that a ball centered at $\bvec{x}_\sfP$ and of radius uniformly comparable to the diameter of $\sfP$ is contained in $\sfP$.

\subsection{Polynomial spaces and $L^2$-orthogonal projectors}

For any $\sfP\in\Mh$ and an integer $\ell\ge 0$, we denote by $\Poly{\ell}(\sfP)$ the space spanned by the restriction to $\sfP$ of three-variate polynomial functions of total degree $\le \ell$, and by $\Poly{0,\ell}(\sfP)$ its subspace made of polynomials whose integral over $\sfP$ vanishes.
Following standard conventions on degrees of polynomials, we moreover set $\Poly{-1}(\sfP)=\{0\}$.
For $F\in\Fh$, $\vPoly{\ell}(F)$ (boldface) is the space of vector-valued polynomials of degree $\le \ell$ on $F$, that are tangent to the face.
For $T\in\Th$, on the other hand, we set $\vPoly{\ell}(T)\coloneq\Poly{\ell}(T)^3$.
We note the following decompositions of the above spaces of vector-valued polynomials (see \cite[Eq.~(3.11)]{Arnold.Falk.ea:06} in the context of polynomial differential forms, and \cite[Table 6.1]{Arnold:18} for the translation in terms of vector proxies):
For all $F\in\Fh$,
\begin{equation*}
  \begin{aligned}
    \vPoly{\ell}(F)
    &= \Goly{\ell}(F) \oplus \cGoly{\ell}(F)
    \quad\text{
      with $\Goly{\ell}(F)\coloneq\GRAD_F\Poly{\ell+1}(F)$
      and $\cGoly{\ell}(F)\coloneq(\bvec{x}-\bvec{x}_F)^\perp\Poly{\ell-1}(F)$
    }
    \\
    &= \Roly{\ell}(F) \oplus \cRoly{\ell}(F)
    \quad\text{
      with $\Roly{\ell}(F)\coloneq\VROT_F\Poly{\ell+1}(F)$
      and $\cRoly{\ell}(F)\coloneq(\bvec{x}-\bvec{x}_F)\Poly{\ell-1}(F)$
    }
  \end{aligned}
\end{equation*}
and, for all $T\in\Th$,
\begin{equation*}
  \begin{aligned}
    \vPoly{\ell}(T)
    &= \Goly{\ell}(T) \oplus \cGoly{\ell}(T)
    \quad\text{
      with $\Goly{\ell}(T)\coloneq\GRAD\Poly{\ell+1}(T)$
      and $\cGoly{\ell}(T)\coloneq(\bvec{x}-\bvec{x}_T)\times \vPoly{\ell-1}(T)$
    }
    \\
    &= \Roly{\ell}(T) \oplus \cRoly{\ell}(T)
    \quad\text{
      with $\Roly{\ell}(T)\coloneq\CURL\Poly{\ell+1}(T)$
      and $\cRoly{\ell}(T)\coloneq(\bvec{x}-\bvec{x}_T)\Poly{\ell-1}(T)$.
    }
  \end{aligned}
\end{equation*}

Given a polynomial (sub)space $\mathcal{X}^\ell(\sfP)$ on $\sfP\in\Mh$, the corresponding $L^2$-orthogonal projector is denoted by $\pi_{\mathcal{X},\sfP}^\ell$, and boldface fonts will be used when the elements of $\mathcal{X}^\ell(\sfP)$ are vector-valued.
Similarly, when considering a complement $\cvec{X}^{\compl,\ell}(\sfP)$ with $\cvec{X}\in\{\cvec{G},\cvec{R}\}$, $\bvec{\pi}_{\cvec{X},\sfP}^{\compl,\ell}$ denotes the $L^2$-orthogonal projector on $\cvec{X}^{\compl,\ell}(\sfP)$.

%------------------------------------------------------------------------------%

\section{DDR complex}\label{sec:ddr.complex}

In this section we briefly recall the DDR complex for a generic polynomial degree $k\ge 0$.
  A complete exposition can be found in \cite[Section 3]{Di-Pietro.Droniou:21*2}; see also \cite{Bonaldi.Di-Pietro.ea:23}, where a more compact and general presentation based on differential forms is provided.
  At the start of Section \ref{sec:cohomology:DDR0}, we present the special case of the lowest-order degree $k=0$, which can be described using simpler formulas.

\subsection{Spaces}
The DDR counterparts of $H^1(\Omega)$, $\Hcurl{\Omega}$, $\Hdiv{\Omega}$, and $L^2(\Omega)$ are respectively defined as follows:
\begin{subequations}\label{eq:ddr.spaces}
\begin{equation}\label{eq:Xgrad.h}
  \Xgrad{k}{h}\coloneq\Big\{
  \begin{aligned}[t]
    \underline{q}_h&=\big((q_T)_{T\in\Th},(q_F)_{F\in\Fh},(q_E)_{E\in\Eh},(q_V)_{V\in\Vh}\big)\st
    \\
    &\text{$q_T\in \Poly{k-1}(T)$ for all $T\in\Th$,
      $q_F\in\Poly{k-1}(F)$ for all $F\in\Fh$,}
    \\
    &\text{$q_E\in\Poly{k-1}(E)$ for all $E\in\Eh$,
      and $q_V\in\Real$ for all $V\in\Vh$}\Big\},
  \end{aligned}
\end{equation}
\begin{equation}\label{eq:Xcurl.h}
  \Xcurl{k}{h}\coloneq\Big\{
  \begin{aligned}[t]
    \uvec{v}_h
    &=\big(
    (\bvec{v}_{\cvec{R},T},\bvec{v}_{\cvec{R},T}^\compl)_{T\in\Th},(\bvec{v}_{\cvec{R},F},\bvec{v}_{\cvec{R},F}^\compl)_{F\in\Fh}, (v_E)_{E\in\Eh}
    \big)\st
    \\
    &\qquad\text{$\bvec{v}_{\cvec{R},T}\in\Roly{k-1}(T)$ and $\bvec{v}_{\cvec{R},T}^\compl\in\cRoly{k}(T)$ for all $T\in\Th$,}
    \\
    &\qquad\text{$\bvec{v}_{\cvec{R},F}\in\Roly{k-1}(F)$ and $\bvec{v}_{\cvec{R},F}^\compl\in\cRoly{k}(F)$ for all $F\in\Fh$,}
    \\
    &\qquad\text{and $v_E\in\Poly{k}(E)$ for all $E\in\Eh$}\Big\},
  \end{aligned}
\end{equation}
\begin{equation}\label{eq:Xdiv.h}
  \Xdiv{k}{h}\coloneq\Big\{
  \begin{aligned}[t]
    \uvec{w}_h
    &=\big((\bvec{w}_{\cvec{G},T},\bvec{w}_{\cvec{G},T}^\compl)_{T\in\Th}, (w_F)_{F\in\Fh}\big)\st
    \\
    &\qquad\text{$\bvec{w}_{\cvec{G},T}\in\Goly{k-1}(T)$ and $\bvec{w}_{\cvec{G},T}^\compl\in\cGoly{k}(T)$ for all $T\in\Th$,}
    \\
    &\qquad\text{and $w_F\in\Poly{k}(F)$ for all $F\in\Fh$}
    \Big\},
  \end{aligned}
\end{equation}
and
\begin{equation}\label{eq:PkTh}
\Poly{k}(\Th)\coloneq\left\{
r_h\in L^2(\Omega)\st\text{$(r_h)_{|T}\in\Poly{k}(T)$ for all $T\in\Th$}
\right\}.
\end{equation}
\end{subequations}
The restrictions of the above spaces and of their elements to a mesh entity $\sfP\in\Mh$ are obtained gathering the polynomial components on $\sfP$ and its boundary, and will be denoted replacing the subscript $h$ by $\sfP$.
So, for example, $\Xgrad{k}{E}$, $\Xgrad{k}{F}$ and $\Xgrad{k}{T}$ are, respectively, the restrictions of $\Xgrad{k}{h}$ to an edge $E\in\Eh$, a face $F\in\Fh$ or an element $T\in\Th$. Similarly, $\Xcurl{k}{F}$ and $\Xcurl{k}{T}$ are the restrictions of $\Xcurl{k}{h}$ to a face $F\in\Fh$ or an element $T\in\Th$.
Following the usual DDR notations, underlined objects represent spaces or vectors having polynomial components, and we use boldface for vector-valued polynomial functions or operators. The sans-serif font is used for ``complete'' differential operators, that only appear in the DDR complex through projections on particular polynomial spaces.

\subsection{Discrete vector calculus operators and potentials}

\subsubsection{Gradient}

For any $E\in\Eh$, the \emph{edge scalar trace} $\trE:\Xgrad{k}{E}\to\Poly{k+1}(E)$ is such that, for all $\underline{q}_E\in\Xgrad{k}{E}$, $\trE\underline{q}_E$ is the unique polynomial in $\Poly{k+1}(E)$ that takes the value $q_V$ in each vertex $V$ of $E$ and satisfies $\lproj{k-1}{E}\trE\underline{q}_E = q_E$.
The \emph{edge gradient} $\GE{k}:\Xgrad{k}{E}\to\Poly{k}(E)$ is defined as:
For all $\underline{q}_E\in\Xgrad{k}{E}$,
\begin{equation}\label{eq:GE}
\int_E \GE{k}\underline{q}_E r_E = -\int_E q_E r_E' +q_{V_2}r_E(\bvec{x}_{V_2})-q_{V_1}r_E(\bvec{x}_{V_1})\qquad\forall r_E\in\Poly{k}(E),
\end{equation}
where $V_1,V_2$ are the two vertices of $E$ numbered according to $\tangent_E$. 

\begin{remark}[Definition of the edge trace and gradient]
  With the notation $q_{\Eh}$ used, e.g., in \cite{Di-Pietro.Droniou:21*1}, we have $\trE\underline{q}_E=(q_{\Eh})_{|E}$.
  Moreover, it is a simple matter to check that $\GE{k}\underline{q}_E=(\trE\underline{q}_E)'$, with the derivative taken in the direction of $\tangent_E$.
\end{remark}

For any $F\in\Fh$, the \emph{face gradient} $\cGF{k}:\Xgrad{k}{F}\to\vPoly{k}(F)$ and the \emph{face scalar trace} $\trF:\Xgrad{k}{F}\to\Poly{k+1}(F)$ are such that, for all $\underline{q}_F\in\Xgrad{k}{F}$,
\begin{equation}\label{eq:cGF}
  \int_F\cGF{k}\underline{q}_F\cdot\bvec{v}_F
  = -\int_F q_F\DIV_F\bvec{v}_F
  + \sum_{E\in\EF}\omega_{FE}\int_E \trE\underline{q}_E~(\bvec{v}_F\cdot\normal_{FE})
  \qquad\forall\bvec{v}_F\in\vPoly{k}(F)
\end{equation}
and
\begin{equation*}%% \label{eq:trF}
  \int_F\trF\underline{q}_F\DIV_F\bvec{v}_F
  = -\int_F\cGF{k}\underline{q}_F\cdot\bvec{v}_F
  + \sum_{E\in\EF}\omega_{FE}\int_E \trE\underline{q}_E~(\bvec{v}_F\cdot\normal_{FE})
  \qquad\forall\bvec{v}_F\in\cRoly{k+2}(F).
\end{equation*}

For all $T\in\Th$, the \emph{element gradient} $\cGT{k}:\Xgrad{k}{T}\to\vPoly{k}(T)$ is defined such that, for all $\underline{q}_T\in\Xgrad{k}{T}$,
\begin{equation}\label{eq:cGT}
  \int_T\cGT{k}\underline{q}_T\cdot\bvec{v}_T
  = -\int_T q_T\DIV\bvec{v}_T
  + \sum_{F\in\FT}\omega_{TF}\int_F\trF\underline{q}_F~(\bvec{v}_T\cdot\normal_F)
  \qquad\forall\bvec{v}_T\in\vPoly{k}(T).
\end{equation}

Finally, the \emph{global discrete gradient} $\uGh{k}:\Xgrad{k}{h}\to\Xcurl{k}{h}$ is obtained collecting the projections of local gradients on the polynomial components of $\Xcurl{k}{h}$: For all $\underline{q}_h\in\Xgrad{k}{h}$,
\begin{equation}\label{eq:uGh}
  \uGh{k}\underline{q}_h\coloneq
  \big(
  ( \Rproj{k-1}{T}\cGT{k}\underline{q}_T,\Rcproj{k}{T}\cGT{k}\underline{q}_T )_{T\in\Th},
  ( \Rproj{k-1}{F}\cGF{k}\underline{q}_F,\Rcproj{k}{F}\cGF{k}\underline{q}_F )_{F\in\Fh},
  ( \GE{k} q_E )_{E\in\Eh}
  \big).
\end{equation}

\subsubsection{Curl}

For all $F\in\Fh$, the \emph{face curl} $\CF{k}:\Xcurl{k}{F}\to\Poly{k}(F)$ and the face tangential trace $\trFt{k}:\Xcurl{k}{F}\to\vPoly{k}(F)$ are such that, for all $\uvec{v}_F\in\Xcurl{k}{F}$,
\begin{equation}\label{eq:CF}
  \int_F\CF{k}\uvec{v}_F~r_F
  = \int_F\bvec{v}_{\cvec{R},F}\cdot\VROT_F r_F
  - \sum_{E\in\EF}\omega_{FE}\int_E v_E~r_F
  \qquad\forall r_F\in\Poly{k}(F)
\end{equation}
and
\begin{multline*}
  \int_F\trFt{k}\uvec{v}_F\cdot(\VROT_F r_F + \bvec{w}_F)
  = \int_F\CF{k}\uvec{v}_F~r_F
  + \sum_{E\in\EF}\omega_{FE}\int_E v_E r_F
  + \int_F\bvec{v}_{\cvec{R},F}^\compl\cdot\bvec{w}_F
  \\
  \forall(r_F,\bvec{w}_F)\in\Poly{0,k+1}(F)\times\cRoly{k}(F).
  %% \label{eq:trFt}
\end{multline*}

For all $T\in\Th$, the \emph{element curl} $\cCT{k}:\Xcurl{k}{T}\to\vPoly{k}(T)$ and the \emph{vector potential reconstruction} $\Pcurl{k}:\Xcurl{k}{T}\to\vPoly{k}(T)$ are defined such that, for all $\uvec{v}_T\in\Xcurl{k}{T}$,
\begin{equation}\label{eq:cCT}
  \int_T\cCT{k}\uvec{v}_T\cdot\bvec{w}_T
  = \int_T\bvec{v}_{\cvec{R},T}\cdot\CURL\bvec{w}_T
  + \sum_{F\in\FT}\omega_{TF}\int_F\trFt{k}\uvec{v}_F\cdot(\bvec{w}_T\times\normal_F)
  \qquad\forall\bvec{w}_T\in\vPoly{k}(T)
\end{equation}
and
\begin{multline*}
  \int_T\Pcurl{k}\uvec{v}_T\cdot(\CURL\bvec{w}_T + \bvec{z}_T)
  = \int_T\cCT{k}\uvec{v}_T\cdot\bvec{w}_T
  - \sum_{F\in\FT}\omega_{TF}\int_F\trFt{k}\uvec{v}_F\cdot(\bvec{w}_T\times\normal_F)
  + \int_T\bvec{v}_{\cvec{R},T}^\compl\cdot\bvec{z}_T
  \\
  \forall(\bvec{w}_T,\bvec{z}_T)\in\cGoly{k+1}(T)\times\cRoly{k}(T).
\end{multline*}

Finally, the \emph{global discrete curl} $\uCh{k}:\Xcurl{k}{h}\to\Xdiv{k}{h}$ is obtained setting, for all $\uvec{v}_h\in\Xcurl{k}{h}$,
\begin{equation}\label{eq:uCh}
  \uCh{k}\uvec{v}_h\coloneq\big(
  ( \Gproj{k-1}{T}\cCT{k}\uvec{v}_T,\Gcproj{k}{T}\cCT{k}\uvec{v}_T )_{T\in\Th},
  ( \CF{k}\uvec{v}_F )_{F\in\Fh}
  \big).
\end{equation}

\subsubsection{Divergence}

For all $T\in\Th$, the \emph{element divergence} $\DT{k}:\Xdiv{k}{T}\to\Poly{k}(T)$ and \emph{vector potential reconstruction} $\Pdiv{k}:\Xdiv{k}{T}\to\vPoly{k}(T)$ are defined by:
For all $\uvec{w}_T\in\Xdiv{k}{T}$,
\begin{equation}\label{eq:DT}
  \int_T\DT{k}\uvec{w}_T~r_T
  = -\int_T\bvec{w}_{\cvec{G},T}\cdot\GRAD r_T
  + \sum_{F\in\FT}\!\omega_{TF}\!\int_F w_F r_T
  \qquad\forall r_T\in\Poly{k}(T)
\end{equation}
and
\begin{multline*}%% \label{eq:Pdiv}
  \int_T\Pdiv{k}\uvec{w}_T\cdot(\GRAD r_T + \bvec{z}_T)
  = -\int_T\DT{k}\uvec{w}_T~r_T
  + \sum_{F\in\FT}\omega_{TF}\int_Fw_F~r_T
  + \int_T\bvec{w}_{\cvec{G},T}^\compl\cdot\bvec{z}_T
  \\
  \forall(r_T,\bvec{z}_T)\in\Poly{0,k+1}(T)\times\cGoly{k}(T).
\end{multline*}

The \emph{global discrete divergence} $\Dh{k}:\Xdiv{k}{h}\to\Poly{k}(\Th)$ is obtained setting, for all $\uvec{w}_h\in\Xdiv{k}{h}$,
\begin{equation}\label{eq:Dh}
  (\Dh{k}\uvec{w}_h)_{|T}\coloneq\DT{k}\uvec{w}_T\qquad\forall T\in\Th,
\end{equation}

\subsection{DDR complex}

The definition of the DDR complex \eqref{eq:ddr.complex} is completed setting, for all $q:\Omega\to\Real$ smooth enough,
\begin{equation}\label{eq:Igradh}
  \Igrad{k}{h} q\coloneq\big(
  (\lproj{k-1}{T} q_{|T})_{T\in\Th},
  (\lproj{k-1}{F} q_{|F})_{F\in\Fh},
  (\lproj{k-1}{E} q_{|E})_{E\in\Eh},
  (q(\bvec{x}_V))_{V\in\Vh}
  \big).
\end{equation}
  A synopsis of the definitions of the DDR spaces and operators is provided in Table \ref{tab:definitions}.
  \begin{table}\centering
    \begin{tabular}{cc|cc}
      \toprule
      DDR operator & Definition & DDR space & Definition \\
      \midrule
      $\Igrad{k}{h}$ & \eqref{eq:Igradh} & $\Xgrad{k}{h}$ & \eqref{eq:Xgrad.h} \\
      $\uGh{k}$ & \eqref{eq:uGh} & $\Xcurl{k}{h}$ & \eqref{eq:Xcurl.h} \\
      $\uCh{k}$ & \eqref{eq:uCh} & $\Xdiv{k}{h}$ & \eqref{eq:Xdiv.h} \\
      $\Dh{k}$ & \eqref{eq:Dh} & $\Poly{k}(\Th)$ & \eqref{eq:PkTh} \\ 
      \bottomrule
    \end{tabular}
    \caption{Definitions of the spaces and operators appearing in the DDR complex \eqref{eq:ddr.complex}.\label{tab:definitions}}
  \end{table}

%------------------------------------------------------------------------------%

\section{Cohomology of the DDR complex}\label{sec:cohomology}

This section contains the proof Theorem \ref{thm:ddr:cohomology} preceded by some preliminary results.

\subsection{Cohomology of the $\DDR{0}$ complex}\label{sec:cohomology:DDR0}

To make the analysis of the cohomology of \DDR{0} easier to read, we explicitly write here the definitions of the spaces and operators in this case.
For $k=0$, the gradient space only has scalar unknowns on the vertices, the curl space on the edges, the divergence space on the faces, and the $L^2$-like space on the elements.
All these spaces can therefore be easily identified as follows:
\begin{subequations}\label{eq:ddr.spaces.k=0}
\begin{equation}\label{eq:Xgrad.h.0}
  \Xgrad{0}{h}=\Big\{\underline{q}_h=(q_V)_{V\in\Vh}\st q_V\in\Real\text{ for all }V\in\Vh\Big\}\simeq \Real^{\Vh},
\end{equation}
\begin{equation}\label{eq:Xcurl.h.0}
  \Xcurl{0}{h}=\Big\{\uvec{v}_h=(v_E)_{E\in\Eh}\st v_E\in\Real\text{ for all }E\in\Eh\Big\}\simeq \Real^{\Eh},
\end{equation}
\begin{equation}\label{eq:Xdiv.h.0}
  \Xdiv{0}{h}=\Big\{\uvec{w}_h= (w_F)_{F\in\Fh}\st w_F\in\Real\text{ for all }F\in\Fh\Big\}\simeq \Real^{\Fh},
\end{equation}
\begin{equation}\label{eq:P0Th}
\Poly{0}(\Th)=\left\{
r_h\in L^2(\Omega)\st (r_h)_{|T}\equiv r_T\in\Real\mbox{ for all } T\in\Th
\right\}\simeq \Real^{\Th}.
\end{equation}
\end{subequations}
The differential operators are linear combinations of values on the boundary of the relevant mesh entities:
For all $(\underline{q}_h,\uvec{v}_h,\uvec{w}_h)\in\Xgrad{0}{h}\times\Xcurl{0}{h}\times\Xdiv{0}{h}$,
\begin{subequations}\label{eq:ddr.operators.k=0}
\begin{alignat}{5}
  \uGh{0}\underline{q}_h
  &=(\GE{0}\underline{q}_E)_{E\in\Eh}
  &&\quad\mbox{with }
  \GE{0}\underline{q}_E
  =\frac{q_{V_2}-q_{V_1}}{|E|}
  \quad\forall E\in\Eh,\label{eq:def.uGh}
  \\
  \uCh{0}\uvec{v}_h
  &=(\CF{0}\uvec{v}_F)_{F\in\Fh}
  &&\quad\mbox{with }
  \CF{0}\uvec{v}_F
  =-\frac{1}{|F|}\sum_{E\in\EF}\omega_{FE}|E|v_E
  \quad\forall F\in\Fh,
  \\
  \Dh{0}\uvec{w}_h
  &=(\DT{0}\uvec{w}_T)_{T\in\Th}
  &&\quad\mbox{with }
  \DT{0}\uvec{w}_T
  =\frac{1}{|T|}\sum_{F\in\EF}\omega_{TF}|F|w_F
  \quad\forall T\in\Th.
\end{alignat}
\end{subequations}
Above, $|\sfP|$ represents the Hausdorff measure of the mesh entity $\sfP\in\Mh$ (its length for an edge, area for a face, volume for an element) and, in \eqref{eq:def.uGh}, we have numbered the vertices $V_1,V_2$ of $E$ so that $\tangent_E$ points from $V_1$ to $V_2$.

\begin{lemma}[Cohomology of the $\DDR{0}$ complex]\label{lem:cohomology:DDR0}
  The cohomology spaces defined by \eqref{eq:cohomology.spaces:discrete} with $k=0$ are isomorphic to the de Rham cohomology spaces \eqref{eq:cohomology.spaces}.
\end{lemma}

\begin{proof}
    The mesh $\Mh$ can be seen as a cellular (CW) complex with vertices in $\Vh$ as 0-cells, edges in $\Eh$ as 1-cells, faces in $\Fh$ as 2-cells, and elements in $\Th$ as 3-cells. With this interpretation, ``discrete de Rham'' maps create isomorphisms between the $\DDR{0}$ spaces (with first space $\Real$ instead of $\{0\}$). Specifically, we have the following diagram:
\begin{equation}\label{eq:DDR0.CW}
  \begin{tikzcd}
    \arrow[<->]{d}{\Id}\Real\arrow{r}{\Igrad{0}{h}} &[1.5em] \Xgrad{0}{h}
    \arrow{r}{\uGh{0}}\arrow{d}{\dRmgrad{h}} &[1.5em] \Xcurl{0}{h}
    \arrow{r}{\uCh{0}}\arrow{d}{\dRmcurl{h}} &[1.5em] \Xdiv{0}{h}
    \arrow{r}{\Dh{0}}\arrow{d}{\dRmdiv{h}} &[1.5em] \Poly{0}(\Th) 
    \arrow{d}{\dRmpoly{h}}\arrow{r} & 0
    \\[2em]
    \Real\arrow{r}{i_\Real} & \Vh^*
    \arrow{r}{\partial_0^*} & \Eh^*
    \arrow{r}{\partial_1^*} & \Fh^*
    \arrow{r}{\partial_2^*} & \Th^*
    \arrow{r} & 0,
  \end{tikzcd}
\end{equation}
in which $\Vh^*$, $\Eh^*$, $\Fh^*$, and $\Th^*$ denote the sets of dual vertices, edges, faces, and volumes, $\partial_i^*$ are the coboundary operators on the cochain complex, $i_\Real$ is the embedding of $\Real$ in $\Vh^*$ as the map $i_\Real(s)(V)=s$ for all $V\in\Vh$, and, recalling \eqref{eq:ddr.spaces.k=0}, the discrete de Rham maps are defined by: 
\begin{alignat*}{3}
\dRmgrad{h}(\underline{q}_h)(V)={}&q_V&\qquad&\forall \underline{q}_h\in\Xgrad{0}{h}\,,\forall V\in\Vh,\\ 
\dRmcurl{h}(\uvec{v}_h)(E)={}&|E|v_E&\qquad&\forall \uvec{v}_h\in\Xcurl{0}{h}\,,\forall E\in\Eh,\\ 
\dRmdiv{h}(\uvec{w}_h)(F)={}&|F|v_F&\qquad&\forall \uvec{w}_h\in\Xdiv{0}{h}\,,\forall F\in\Fh,\\
\dRmpoly{h}(r_h)(T)={}&|T|r_T&\qquad&\forall r_h\in\Poly{0}(\Th)\,,\forall T\in\Th.
\end{alignat*}
It is trivial to verify that these de Rham maps are isomorphisms, and that the diagram \eqref{eq:DDR0.CW} is commutative (the latter property simply consists in using \eqref{eq:ddr.operators.k=0} and translating the definitions of the coboundary operators). This proves that the $\DDR{0}$ complex is isomorphic to the cochain complex, and thus has the same cohomology spaces.
Since, in turn, these cohomology spaces are isomorphic to the de Rham cohomology spaces \eqref{eq:cohomology.spaces} (see, e.g., \cite{Warner:83}), this concludes the proof.
\end{proof}

\subsection{Framework for isomorphisms in cohomology}\label{sec:framework.cohomology}

The proof of Theorem \ref{thm:ddr:cohomology} hinges on the following result, which provides a simplified version of the framework developed in \cite{Di-Pietro.Droniou:22}.

\begin{theorem}[Complexes with isomorphic cohomologies]\label{thm:framework}
  Consider two sequences of spaces and operators $(\Xg{i},\dg{i})_i$ and $(\tXg{i},\tdg{i})_i$ connected by the graded maps $(\Rg{i})_i$ and $(\Eg{i})_i$, respectively called \emph{reductions} and \emph{extensions}, as depicted in the following diagram:
  \begin{equation}\label{eq:general.diagram}
    \begin{tikzcd}
      \cdots\arrow{r} & \Xg{i}
      \arrow{r}{\dg{i}}\arrow[bend left]{d}{\Rg{i}} &[3em] \Xg{i+1}
      \arrow{r}{}\arrow[bend left]{d}{\Rg{i+1}} & \cdots\\[2em]
      \cdots\arrow{r} & \tXg{i}
      \arrow{r}{\tdg{i}}\arrow[bend left]{u}{\Eg{i}} &[3em] \tXg{i+1}
      \arrow{r}{}\arrow[bend left]{u}{\Eg{i+1}} & \cdots
    \end{tikzcd}
  \end{equation}
  Further assume that $(\Xg{i},\dg{i})_i$ is a complex and that the following properties hold:
  \begin{itemize}
  \item[\textbf{(C1)}] The reductions are left-inverses of the extensions, i.e., $\Rg{i}\Eg{i}=\Id_{\tXg{i}}$ for all $i$;
  \item[\textbf{(C2)}] $(\Eg{i+1}\Rg{i+1}-\Id_{\Xg{i+1}})(\ker\dg{i+1})\subset\Image\dg{i}$ for all $i$;
  \item[\textbf{(C3)}] The extensions and reductions are cochain maps.
  \end{itemize}
  Then, $(\tXg{i},\tdg{i})_i$ is a complex and its cohomology is isomorphic to that of $(\Xg{i},\dg{i})_i$; more precisely, the extensions and reductions induce reciprocal isomorphisms in cohomology.
\end{theorem}

\begin{proof}
  By \textbf{(C3)} and \textbf{(C1)} (applied to $i+1$) we have $\Rg{i+1}\dg{i}\Eg{i} = \Rg{i+1}\Eg{i+1}\tdg{i}=\tdg{i}$ for all $i$. We deduce from this fact that $(\tXg{i},\tdg{i})_i$ is a complex by invoking \textbf{(C3)} and the complex property of $(\Xg{i},\dg{i})_i$ to write
  \[
  \tdg{i+1}\tdg{i}=(\Rg{i+2}\underbrace{\dg{i+1}\Eg{i+1}}_{=\Eg{i+2}\tdg{i+1}})(\Rg{i+1}\dg{i}\Eg{i})=\Rg{i+2}\Eg{i+2}\underbrace{\tdg{i+1}\Rg{i+1}}_{=\Rg{i+2}\dg{i+1}}\dg{i}\Eg{i}=\Rg{i+2}\Eg{i+2}\Rg{i+2}\underbrace{\dg{i+1}\dg{i}}_{=0}\Eg{i}=0.
  \]

  Since extensions and reductions are cochain maps, they induce mappings $[\Eg{i}]$ and $[\Rg{i}]$ between the cohomology groups of the complexes. Property \textbf{(C1)} then gives $[\Id_{\tXg{i}}]=[\Rg{i}\Eg{i}]=[\Rg{i}][\Eg{i}]$. On the other hand, \textbf{(C2)} applied to $i$ instead of $i+1$ implies $[\Eg{i}\Rg{i}-\Id_{\Xg{i}}]=0$ and thus $[\Eg{i}][\Rg{i}]=[\Id_{\Xg{i}}]$, which concludes the proof that $[\Eg{i}]$ and $[\Rg{i}]$ are reciprocal isomorphisms.
\end{proof}

Lemma \ref{lem:cohomology:DDR0} proves that the cohomology of $\DDR{0}$ is isomorphic to the continuous de Rham cohomology. To prove Theorem \ref{thm:ddr:cohomology}, we therefore only need to show that, for $k\ge 1$, the cohomologies of $\DDR{k}$ and $\DDR{0}$ are isomorphic; this will be done by constructing reduction and extension maps between these two complexes and using Theorem \ref{thm:framework}. In Section \ref{sec:red.ext} we design these maps and show that they satisfy \textbf{(C1)} and \textbf{(C3)}, while Section \ref{sec:proof.C2} establishes \textbf{(C2)}.

\subsection{Reduction and extension cochain maps, and proofs of \textbf{(C1)} and \textbf{(C3)}}\label{sec:red.ext}

The specialisation of diagram \eqref{eq:general.diagram} linking the $\DDR{k}$ and $\DDR{0}$ complexes takes the following form:
\begin{equation}\label{eq:DDRk.DDR0}
\begin{tikzcd}
  \arrow[<->]{d}{\Id}\Real\arrow{r}{\Igrad{k}{h}} &[1em] \Xgrad{k}{h}
  \arrow{r}{\uGh{k}}\arrow[bend left]{d}{\Rgrad{h}} &[2em] \Xcurl{k}{h}
  \arrow{r}{\uCh{k}}\arrow[bend left]{d}{\Rcurl{h}} &[2em] \Xdiv{k}{h} 
  \arrow{r}{\Dh{k}}\arrow[bend left]{d}{\Rdiv{h}} &[1em] \Poly{k}(\Th) 
  \arrow[bend left]{d}{\lproj{0}{h}}\arrow{r} & 0
  \\[2em]
  \Real\arrow{r}{\Igrad{0}{h}} & \Xgrad{0}{h}
  \arrow{r}{\uGh{0}}\arrow[bend left]{u}{\Egrad{h}} &[2em] \Xcurl{0}{h}
  \arrow{r}{\uCh{0}}\arrow[bend left]{u}{\Ecurl{h}} &[2em] \Xdiv{0}{h}
  \arrow{r}{\Dh{0}}\arrow[bend left]{u}{\Ediv{h}}  &[1em] \Poly{0}(\Th)
  \arrow[bend left]{u}{i}
  \arrow{r} & 0,
\end{tikzcd}
\end{equation}
where $i$ is the natural inclusion and $\lproj{0}{h}$ is the $L^2$-orthogonal projection on $\Poly{0}(\Th)$.
In the rest of this section we provide the definitions of the reductions $\Rgrad{h}$, $\Rcurl{h}$, $\Rdiv{h}$, and of the extensions $\Egrad{h}$, $\Ecurl{h}$, $\Ediv{h}$, and show that they satisfy \textbf{(C1)} and \textbf{(C3)}.

\subsubsection{Reductions}\label{sec:reductions.extensions:R}

The reduction maps are naturally obtained taking the $L^2$-orthogonal projections of the components attached to the lowest-dimensional mesh entities in each space, i.e.,
\begin{alignat*}{4}%\label{eq:R.grad.h}
  \Rgrad{h}\underline{q}_h &\coloneq (q_V)_{V\in\Vh} &\qquad& \forall\underline{q}_h\in\Xgrad{k}{h},
  \\ %\label{eq:R.curl.h}
  \Rcurl{h}\uvec{v}_h &\coloneq (\lproj{0}{E} v_E)_{E\in\Eh} &\qquad& \forall\uvec{v}_h\in\Xcurl{k}{h},
  \\ %\label{eq:R.div.h}
  \Rdiv{h}\uvec{w}_h &\coloneq (\lproj{0}{F} w_F)_{F\in\Fh} &\qquad& \forall\uvec{w}_h\in\Xdiv{k}{h}.
\end{alignat*}
By definition of the discrete differential and reduction operators, one can easily check that $\Rgrad{h}\Igrad{k}{h}=\Igrad{0}{h}$, $\Rcurl{h}\uGh{k} = \uGh{0}\Rgrad{h}$, $\Rdiv{h}\uCh{k} = \uCh{0}\Rcurl{h}$, and $\lproj{0}{h}\Dh{k} = \Dh{0}\Rdiv{h}$, showing that the reductions yield a cochain map, as required by \textbf{(C3)}.

\subsubsection{Extensions}\label{sec:reductions.extensions:E}

The definition of the extension maps is, on the other hand, more subtle.
Concerning the gradient space, we set, for all $\underline{q}_h\in\Xgrad{0}{h}$,
\begin{subequations}\label{eq:E.grad.h}
  \begin{equation}
    \Egrad{h}\underline{q}_h
    \coloneq\big(
    (\EPoly{T}\underline{q}_T)_{T\in\Th},
    (\EPoly{F}\underline{q}_F)_{F\in\Fh},
    (\EPoly{E}\underline{q}_E)_{E\in\Eh},
    (q_V)_{V\in\Vh}
    \big)
  \end{equation}
  with, for all $E\in\Eh$, $\EPoly{E}\underline{q}_E\in\Poly{k-1}(E)$ such that
  \begin{equation}\label{eq:E.grad.h:EqE}
    \int_E\EPoly{E}\underline{q}_E~r_E'
    = -\int_E\GE{0}\underline{q}_E~r_E
    + q_{V_2}~r_E(\bvec{x}_{V_2})
    - q_{V_1}~r_E(\bvec{x}_{V_1})
    \qquad\forall r_E\in\Poly{k}(E),
  \end{equation}
  for all $F\in\Fh$, $\EPoly{F}\underline{q}_F\in\Poly{k-1}(F)$ such that
  \begin{multline}\label{eq:E.grad.h:EqF}
    \int_F\EPoly{F}\underline{q}_F~\DIV_F\bvec{v}_F
    = -\int_F\cGF{0}\underline{q}_F\cdot\bvec{v}_F
    + \sum_{E\in\EF}\omega_{FE}\int_E\trE\Egrad{E}\underline{q}_E~(\bvec{v}_F\cdot\normal_{FE})
    \\
    \forall\bvec{v}_F\in\cRoly{k}(F),
  \end{multline}
  and, for all $T\in\Th$, $\EPoly{T}\underline{q}_T\in\Poly{k-1}(T)$ such that
  \begin{multline}\label{eq:E.grad.h:EqT}
    \int_T\EPoly{T}\underline{q}_T~\DIV\bvec{v}_T
    = -\int_T\cGT{0}\underline{q}_T\cdot\bvec{v}_T
    + \sum_{F\in\FT}\omega_{TF}\int_F\trF\Egrad{F}\underline{q}_F~(\bvec{v}_T\cdot\normal_F)
    \\
    \forall\bvec{v}_T\in\cRoly{k}(T).
  \end{multline}
\end{subequations}
In \eqref{eq:E.grad.h:EqF} and \eqref{eq:E.grad.h:EqT}, we have respectively introduced the notations $\Egrad{E}\underline{q}_E\coloneq\big(\EPoly{E}\underline{q}_E,(q_V)_{V\in\VE}\big)$ and  $\Egrad{F}\underline{q}_F\coloneq\big(\EPoly{F}\underline{q}_F, (\EPoly{E}\underline{q}_E)_{E\in\EF}, (q_V)_{V\in\VF}\big)$ (with $\VE$, resp. $\VF$, denoting the set of vertices of $E$, resp. $F$).

\begin{remark}[Test functions in the definition of the extension operators]
  To properly define $\EPoly{E}$, we should only consider in \eqref{eq:E.grad.h:EqE} test functions $r_E\in\Poly{0,k}(E)$ (as the derivative
  is an isomorphism $\Poly{0,k}(E)\to\Poly{k-1}(E)$); however, we note that \eqref{eq:E.grad.h:EqE} is also satisfied for constant $r_E$ since the right-hand side then vanishes by definition of $\GE{0}$. This is why we can actually consider test functions in the entire space $\Poly{k}(E)$.
  Similar considerations hold for \eqref{eq:E.curl.h:EvF} and \eqref{eq:E.div.h:EwT} below.
\end{remark}

The extension operator from $\Xcurl{0}{h}$ to $\Xcurl{k}{h}$ is such that, for all $\uvec{v}_h\in\Xcurl{0}{h}$,
\begin{subequations}\label{eq:E.curl.h}
  \begin{equation}
    \Ecurl{h}\uvec{v}_h
    \coloneq\big(
    (\ERoly{T}\uvec{v}_T,\Rcproj{k}{T}\Pcurl{0}\uvec{v}_T)_{T\in\Th},
    (\ERoly{F}\uvec{v}_F,\Rcproj{k}{F}\trFt{0}\uvec{v}_F)_{F\in\Fh},
    (v_E)_{E\in\Eh}
    \big),
  \end{equation}
  where, for all $F\in\Fh$, $\ERoly{F}\uvec{v}_F\in\Roly{k-1}(F)$ is such that
  \begin{equation}\label{eq:E.curl.h:EvF}
    \int_F\ERoly{F}\uvec{v}_F\cdot\VROT_F r_F
    = \int_F\CF{0}\uvec{v}_F~r_F
    + \sum_{E\in\EF}\omega_{FE}\int_E v_E~r_F
    \qquad\forall r_F\in\Poly{k}(F),
  \end{equation}
  and, for all $T\in\Th$, $\ERoly{T}\uvec{v}_T\in\Roly{k-1}(T)$ is such that
  \begin{multline}\label{eq:E.curl.h:EvT}
    \int_T\ERoly{T}\uvec{v}_T\cdot\CURL\bvec{w}_T
    = \int_T\cCT{0}\uvec{v}_T\cdot\bvec{w}_T
    - \sum_{F\in\FT}\omega_{TF}\int_F\trFt{k}\Ecurl{F}\uvec{v}_F\cdot(\bvec{w}_T\times\normal_F)
    \\
    \forall\bvec{w}_T\in\cGoly{k}(T),
  \end{multline}
\end{subequations}
with
\begin{equation}\label{eq:Ecurl.F}
  \Ecurl{F}\uvec{v}_F\coloneq\big(\ERoly{F}\uvec{v}_F,\Rcproj{k}{F}\trFt{0}\uvec{v}_F,(v_E)_{E\in\EF}\big).
\end{equation}
\smallskip

The extension operator from $\Xdiv{0}{h}$ to $\Xdiv{k}{h}$ is such that, for all $\uvec{w}_h\in\Xdiv{0}{h}$,
\begin{subequations}\label{eq:E.div.h}
  \begin{equation}
    \Ediv{h}\uvec{w}_h
    = \big(
    (\EGoly{T}\uvec{w}_T,\Gcproj{k}{T}\Pdiv{0}\uvec{w}_T)_{T\in\Th},
    (w_F)_{F\in\Fh}
    \big),
  \end{equation}
  where, for all $T\in\Th$, $\EGoly{T}\uvec{w}_T\in\Goly{k-1}(T)$ is such that
  \begin{equation}\label{eq:E.div.h:EwT}
    \int_T\EGoly{T}\uvec{w}_T\cdot\GRAD r_T
    = -\int_T\DT{0}\uvec{w}_T~r_T
    + \sum_{F\in\FT}\omega_{TF}\int_Fw_F~r_T
    \qquad\forall r_T\in\Poly{k}(T).
  \end{equation}
\end{subequations}

A simple inspection of the reductions and extensions shows that \textbf{(C1)} holds for these maps, that is:
  \begin{equation}\label{eq:left-inverse.E}
    \begin{alignedat}{4}
      \Rgrad{h}\Egrad{h} &= \Id_{\Xgrad{0}{h}},
      &\qquad \Rcurl{h}\Ecurl{h}&=\Id_{\Xcurl{0}{h}},\\
      \Rdiv{h}\Ediv{h}&=\Id_{\Xdiv{0}{h}},
      &\qquad \lproj{0}{h}i&=\Id_{\Poly{0}(\Th)}.
    \end{alignedat}
  \end{equation}
The fact that the above-defined extensions form a cochain map (i.e., they satisfy \textbf{(C3)}) requires, on the other hand, a detailed proof, provided in the following lemma.

\begin{lemma}[Cochain map properties for the extensions]\label{lem:cochain.map:E}
  The extensions are cochain maps, that is:
  \begin{alignat}{4}
    \label{eq:Igradk.E=Igrad0}
    \Igrad{k}{h}C
    ={}& \Egrad{h}\Igrad{0}{h}C&\qquad&\forall C\in\Real,\\  
    \label{eq:Gk.E=G0}
    \uGh{k}\Egrad{h}\underline{q}_h
    ={}& \Ecurl{h}\uGh{0}\underline{q}_h&\qquad&\forall\underline{q}_h\in\Xgrad{0}{h},\\
    \label{eq:Ck.E=C0}
    \uCh{k}\Ecurl{h}\uvec{v}_h
    ={}& \Ediv{h}\uCh{0}\uvec{v}_h&\qquad&\forall\uvec{v}_h\in\Xcurl{0}{h},\\
    \label{eq:Dk.E=D0}
    \Dh{k}\Ediv{h}\uvec{w}_h
    ={}& \Dh{0}\uvec{w}_h&\qquad&\forall\uvec{w}_h\in\Xdiv{0}{h}.
  \end{alignat}
\end{lemma}

\begin{proof}
  \emph{(i) Proof of \eqref{eq:Igradk.E=Igrad0}}. This amounts to checking that, for any mesh entity $\sfP\in\Th\cup\Fh\cup\Eh$, $\EPoly{\sfP}\Igrad{0}{\sfP}C=\lproj{k-1}{\sfP}C$ for all $C\in\Real$, which is straightforward from the definition of the extension operators and the polynomial consistency of the edge and face gradients and scalar traces.
  \medskip\\  
  \emph{(ii) Proof of \eqref{eq:Gk.E=G0}}.
  Let $\underline{q}_h\in\Xgrad{0}{h}$.
  Combining the definition \eqref{eq:GE} of $\GE{k}$ with the definition \eqref{eq:E.grad.h:EqE} of $\EPoly{E}$ immediately gives $\GE{k}\Egrad{E}\underline{q}_E=\GE{0}\underline{q}_E$, which shows the equality of the edge components in \eqref{eq:Gk.E=G0}.
  Applying the definition \eqref{eq:cGF} of $\cGF{k}$ with $\bvec{v}_F\in\cRoly{k}(F)\subset\vPoly{k}(F)$ and invoking the definition \eqref{eq:E.grad.h:EqF} of $\EPoly{F}$ gives $\Rcproj{k}{F}\cGF{k}\Egrad{F}\underline{q}_F=\Rcproj{k}{F}\cGF{0}\underline{q}_F=\Rcproj{k}{F}\trFt{0}\uGF{0}\underline{q}_F$, where the second equality comes from \cite[Eq.~(3.26)]{Di-Pietro.Droniou:21*1}. The same arguments, based on \eqref{eq:cGT}, \eqref{eq:E.grad.h:EqT}, and \cite[Eq.~(4.29)]{Di-Pietro.Droniou:21*1}, give $\Rcproj{k}{T}\cGT{k}\Egrad{T}\underline{q}_T=\Rcproj{k}{T}\Pcurl{0}\uGT{0}\underline{q}_T$. This establishes the equality of the components in $\cRoly{k}(\sfP)$, $\sfP\in\Th\cup\Fh$, in \eqref{eq:Gk.E=G0}.

  We next show that, for all $F\in\Fh$,
  \begin{equation}\label{eq:Gk.E=G0:Rcproj.F}
    \Rproj{k-1}{F}\cGF{k}\Egrad{F}\underline{q}_F
    = \ERoly{F}\uGF{0}\underline{q}_F.
  \end{equation}
  Apply the definition \eqref{eq:cGF} of $\cGF{k}$ to $\bvec{v}_F=\VROT_F r_F$ for some $r_F\in\Poly{k}(F)$ and use $\DIV_F\VROT_F=0$ together with $\VROT_F r_F\cdot\normal_{FE}=-(r_F)_{|E}'$, the derivative being taken in the direction of $\tangent_E$ (see \cite[Eq.~(4.20)]{Di-Pietro.Droniou.ea:20}) to write
  \begin{align*}
    \int_F \cGF{k}\Egrad{F}\underline{q}_F\cdot\VROT_F r_F={}&-\sum_{E\in\EF}\omega_{FE}\int_E\trE\Egrad{E}\underline{q}_E\,(r_F)_{|E}'\\
    ={}&-\sum_{E\in\EF}\omega_{FE}\int_E\lproj{k-1}{E}(\trE\Egrad{E}\underline{q}_E)\,(r_F)_{|E}'\\
    ={}&\sum_{E\in\EF}\omega_{FE}\int_E\GE{0}\underline{q}_E\,r_F
    -\cancel{%
      \sum_{E\in\EF}\omega_{FE}\left(
      q_{V_{2,E}}r_F(\bvec{x}_{V_{2,E}})-q_{V_{1,E}}r_F(\bvec{x}_{V_{1,E}})
      \right),
    }
  \end{align*}
  where the introduction of the projector in the second line is justified by $(r_F)_{|E}'\in\Poly{k-1}(E)$, and the third line follows from $\lproj{k-1}{E}(\trE\Egrad{E}\underline{q}_E)=\EPoly{E}\underline{q}_E$ (by definition of $\trE$) and the definition \eqref{eq:E.grad.h:EqE} of $\EPoly{E}$ (we have added the index $E$ in the vertices to clearly show that they are related to each edge in the sum); the final cancellation is obtained by noticing that each vertex of $F$ appears twice in the sum with opposite orientations $\omega_{FE}$. We then apply the definition \eqref{eq:E.curl.h:EvF} of $\ERoly{F}$ to $\uvec{v}_F=\uGF{0}\underline{q}_F$ together with the complex property $\CF{0}\uGF{0}=0$ to deduce
  \[
  \int_F \cGF{k}\Egrad{F}\underline{q}_F\cdot\VROT_F r_F=\int_F \ERoly{F}\uGF{0}\underline{q}_F\cdot\VROT_F r_F,
  \]
  which concludes the proof of \eqref{eq:Gk.E=G0:Rcproj.F}.
  Together with the equality of the components in $\cRoly{k}(F)$ and on the edges, this shows that
  \begin{equation}\label{eq:GFk.E=GF0}
    \uGF{k}\Egrad{F}\underline{q}_F=\Ecurl{F}\uGF{0}\underline{q}_F\qquad\forall F\in\Fh.
  \end{equation}  
  
  Let us now take $T\in\Th$ and let us prove the equality of the components in $\Roly{k-1}(T)$ in \eqref{eq:Gk.E=G0}, i.e.,
  \begin{equation}\label{eq:Gk.E=G0:Rcproj.T}
      \Rproj{k-1}{T}\cGT{k}\Egrad{T}\underline{q}_T
      = \ERoly{T}\uGT{0}\underline{q}_T.
  \end{equation}%
  For all $\bvec{w}_T\in\cGoly{k}(T)$, using the link between element and face discrete gradients together with the property $\trFt{k}\uGF{k}=\cGF{k}$ of the tangential trace (see \cite[Proposition 1 and Eq.~(3.26)]{Di-Pietro.Droniou:21*1}), we have
  \begin{align*}
    \int_T \cGT{k}\Egrad{T}\underline{q}_T\cdot\CURL\bvec{w}_T
    ={}&-\sum_{F\in\FT}\omega_{TF}\int_F \trFt{k}\uGF{k}\Egrad{F}\underline{q}_F\cdot(\bvec{w}_T\times\normal_F)\\
    ={}&-\sum_{F\in\FT}\omega_{TF}\int_F \trFt{k}\Ecurl{F}\uGF{0}\underline{q}_F\cdot(\bvec{w}_T\times\normal_F)\\
    ={}&\int_T \ERoly{T}\uGT{0}\underline{q}_T\cdot\CURL\bvec{w}_T,
  \end{align*}
  where the second equality follows from \eqref{eq:GFk.E=GF0}, and the third one from the definition \eqref{eq:E.curl.h:EvT} of $\ERoly{T}$ together with the complex property $\uCT{0}\uGT{0}=0$ (which implies $\cCT{0}\uGT{0}=\bvec{0}$ since $\Pdiv{0}\uCT{0}=\cCT{0}$ by \cite[Eq.~(4.30)]{Di-Pietro.Droniou:21*1}). This concludes the proof of \eqref{eq:Gk.E=G0:Rcproj.T}.
  \medskip\\  
  \emph{(iii) Proof of \eqref{eq:Ck.E=C0}}.
  Let $\uvec{v}_h\in\Xcurl{0}{h}$ and $F\in\Fh$.
  The definitions \eqref{eq:CF} of $\CF{k}$ and \eqref{eq:Ecurl.F} of $\Ecurl{F}$ show that 
  \begin{equation}\label{eq:CFk.E=CF0}
    \CF{k}\Ecurl{F}\uvec{v}_F=\CF{0}\uvec{v}_F,
  \end{equation}
  which proves the equality of the face components in \eqref{eq:Ck.E=C0}. Take now $T\in\Th$ and consider the component in $\cGoly{k}(T)$. Applying the definitions \eqref{eq:cCT} of $\cCT{k}$ and \eqref{eq:E.curl.h:EvT} of $\Ecurl{T}$ to a generic $\bvec{w}_T\in\cGoly{k}(T)\subset\vPoly{k}(T)$ yields $\Gcproj{k}{T}\cCT{k}\Ecurl{T}\uvec{v}_T=\Gcproj{k}{T}\cCT{0}\uvec{v}_T=\Gcproj{k}{T}\Pdiv{0}\uCT{0}\uvec{v}_T$, where the second equality is obtained applying {\cite[Eq.~(4.30)]{Di-Pietro.Droniou:21*1}}. It remains to show the equality of the components in $\Goly{k-1}(T)$ in \eqref{eq:Ck.E=C0}, i.e.,
    \begin{equation}\label{eq:Ck.E=C0:Gproj.T}
      \Gproj{k-1}{T}\cCT{k}\Ecurl{T}\uvec{v}_T = \EGoly{T}\uCT{0}\uvec{v}_T.
    \end{equation}
We use the link between element and face discrete curls \cite[Proposition 4]{Di-Pietro.Droniou:21*1} together with \eqref{eq:CFk.E=CF0} to write, for all $r_T\in\Poly{k}(T)$,
  \[
  \int_T \cCT{k}\Ecurl{T}\uvec{v}_T\cdot\GRAD r_T=\sum_{F\in\FT}\omega_{TF}\int_F \CF{k}\Ecurl{F}\uvec{v}_Fr_T=\sum_{F\in\FT}\omega_{TF}\int_F \CF{0}\uvec{v}_Fr_T.
  \]
  Invoking then the definition \eqref{eq:E.div.h:EwT} of $\EGoly{T}$ with $\uvec{w}_T=\uCT{0}\uvec{v}_T$ and using the complex property $\DT{0}\uCT{0}=0$, we infer
  \[
  \int_T \cCT{k}\Ecurl{T}\uvec{v}_T\cdot\GRAD r_T=\int_T\EGoly{T}\uCT{0}\uvec{v}_T\cdot\GRAD r_T,
  \]
  which concludes the proof of \eqref{eq:Ck.E=C0:Gproj.T}.
  \medskip\\  
  \emph{(iv) Proof of \eqref{eq:Dk.E=D0}}.
  Let $\uvec{w}_h\in\Xdiv{0}{h}$.
  For all $T\in\Th$, apply the definitions \eqref{eq:DT} of $\DT{k}$ and \eqref{eq:E.div.h:EwT} of $\EGoly{T}$ to get $\DT{k}\Ediv{T}\uvec{w}_T=\DT{0}\uvec{w}_T$.
  
\end{proof}

\subsection{Proof of \textbf{(C2)}}\label{sec:proof.C2}

To conclude the proof of Theorem \ref{thm:ddr:cohomology}, we need to show that the reduction and extension maps satisfy \textbf{(C2)}, which is the purpose of the following lemma.

\begin{lemma}[Property \textbf{(C2)}]\label{lem:propC2}
The maps in \eqref{eq:DDRk.DDR0} satisfy the following properties:
  \begin{subequations}\label{eq:left-inverse.R}
    \begin{gather}\label{eq:left-inverse.R:grad}
      \text{%         
        For all $\underline{q}_h\in\Ker\uGh{k}$, there exists $C\in\Real$ such that
        $\Egrad{h}\Rgrad{h}\underline{q}_h-\underline{q}_h = \Igrad{k}{h} C$,
      }
      \\\label{eq:left-inverse.R:curl}
      \text{%
        For all $\uvec{v}_h\in\Ker\uCh{k}$, there exists $\underline{q}_h\in\Xgrad{k}{h}$ such that
        $\Ecurl{h}\Rcurl{h}\uvec{v}_h - \uvec{v}_h = \uGh{k}\underline{q}_h$,
      }
      \\\label{eq:left-inverse.R:div}
      \text{%
        For all $\uvec{w}_h\in\Ker\Dh{k}$, there exists $\uvec{v}_h\in\Xcurl{k}{h}$ such that
        $\Ediv{h}\Rdiv{h}\uvec{w}_h - \uvec{w}_h = \uCh{k}\uvec{v}_h$,
      }
      \\\label{eq:left-inverse.R:Pk}
      \text{%
        For all $r_h\in\Poly{k}(\Th)$, there exists $\uvec{w}_h\in\Xdiv{k}{h}$ such that
        $\lproj{0}{h}r_h-r_h = \Dh{k}\uvec{w}_h$.
      }
    \end{gather}
  \end{subequations}%
\end{lemma}

\begin{proof}  
  \emph{(i)} \emph{Proof of \eqref{eq:left-inverse.R:grad}.}
  For all $\underline{q}_h\in\Xgrad{k}{h}$, $\uGh{k}\underline{q}_h = \uvec{0}$ implies $\uGT{k}\underline{q}_T = \uvec{0}$ for all $T\in\Th$ which, by exactness of the local DDR sequence (see the proof of \cite[Theorem 2]{Di-Pietro.Droniou:21*2} applied to $T$, which is topologically trivial by assumption), implies the existence of $Q_T\in\Real$ such that $\underline{q}_T=\Igrad{k}{T}Q_T$ (where $\Igrad{k}{T}$ is the restriction to $\Xgrad{k}{T}$ of $\Igrad{k}{h}$). The definition \eqref{eq:Igradh} of the interpolator then shows that $q_V=Q_T$ for all vertex $V\in\VT$. We deduce that $Q_T=Q_{T'}$ whenever $T,T'\in\Th$ share a vertex which yields, by connectedness of $\Omega$, the existence of $Q\in\Real$ such that  $Q_T = Q$ for all $T\in\Th$.
  As a consequence, $\underline{q}_h=\Igrad{k}{h}Q$.
  Using this fact along with the cochain property of the extension and reduction, we infer that $\Egrad{h}\Rgrad{h}\underline{q}_h=\Egrad{h}\Igrad{0}{h}Q=\Igrad{k}{h}Q$, and thus that $\Egrad{ h}\Rgrad{h}\underline{q}_h - \underline{q}_h=\Igrad{k}{h}0$.
  This proves \eqref{eq:left-inverse.R:grad}.
  \medskip\\
  \emph{(ii)} \emph{Proof of \eqref{eq:left-inverse.R:curl}.} 
  We start by noticing that, since the reductions and extensions are cochain maps, if $\uvec{v}_h\in\Xcurl{k}{h}$ is such that $\uCh{k}\uvec{v}_h = \uvec{0}$ then
  \[
  \uCh{k}\Ecurl{h}\Rcurl{h}\uvec{v}_h=\Ediv{h}\uCh{0}\Rcurl{h}\uvec{v}_h=\Ediv{h}\Rdiv{h}\uCh{k}\uvec{v}_h=\uvec{0}.
  \]
  Hence, $\Ecurl{h}\Rcurl{h}\uvec{v}_h - \uvec{v}_h\in\Ker\uCh{k}$, and the exactness of the local DDR complex implies, for all $T\in\Th$, the existence of $\underline{q}_T^T\in\Xgrad{k}{T}$ such that
  \begin{equation}\label{eq:P.complex.ER:curl:T}
    \uvec{z}_T\coloneq\Ecurl{T}\Rcurl{T}\uvec{v}_T - \uvec{v}_T = \uGT{k}\underline{q}_T^T.
  \end{equation}
  We then have to check that the $\underline{q}_T^T$, $T\in\Th$, can be glued together to form an element of $\Xgrad{k}{h}$.
  Let $T\in\Th$ and notice that, by definition of the reduction and extension operators, for all $E\in\ET$, $\GE{k}\underline{q}_E^T=z_E=\lproj{0}{E}v_E-v_E$ has a zero integral over $E$.
  The definition \eqref{eq:GE} of $\GE{k}$ thus shows that $\underline{q}_T^T$ takes the same value at the vertices of $E$; since this holds for any edge of $T$ and the boundary of $T$ is connected, this implies the existence of $C_T\in\Real$ such that $q_V^T = C_T$ for all $V\in\VT$.
  By the substitution $\underline{q}_T^T\gets\underline{q}_T^T - \Igrad{k}{T} C_T$, which leaves \eqref{eq:P.complex.ER:curl:T} unaltered since $\uGT{k}\Igrad{k}{T} C_T= \uvec{0}$ by the complex property, we obtain local vectors $\underline{q}_T^T$, $T\in\Th$, that vanish (hence match) at mesh vertices.
  We next notice that, for all $E\in\Eh$ and any $T$ such that $E\in\ET$, using \eqref{eq:P.complex.ER:curl:T} along with the definition \eqref{eq:GE} of the edge gradient and the fact that the vertex values of $\underline{q}_T^T$ vanish,
  \begin{equation}\label{eq:cond.qE}
    \int_E z_E~r_E
    = \int_E\GE{k}\underline{q}_E^T~r_E
    = -\int_E q_E^T r_E'
    \qquad\forall r_E\in\Poly{0,k}(E).
  \end{equation}
  The relation \eqref{eq:cond.qE} implies that $q_E^T$ is in fact independent of $T$, and thus that there exists $q_E\in\Poly{k-1}(E)$ such that $q_E^T = q_E$ for all $T\in\Th$ such that $E\in\ET$.
  We therefore set $\underline{q}_E\coloneq (q_E, (0)_{V\in\VE})$ for all $E\in\Eh$, where we remind the reader that $\VE$ collects the vertices of $E$.

  Having proved the single-valuedness of the $\underline{q}_T^T$, $T\in\Th$, on the mesh edge skeleton, we next notice that, for all $F\in\Fh$ and all $T\in\Th$ such that $F\in\FT$, \eqref{eq:P.complex.ER:curl:T} followed by the definition \eqref{eq:cGF} of $\cGF{k}$ implies, for all $\bvec{w}_F\in\cRoly{k}(F)$,
  \[
  \int_F \bvec{z}_{\cvec{R},F}^\compl\cdot\bvec{w}_F
  = \int_F \cGF{k}\underline{q}_{F}^T\cdot\bvec{w}_F
  = -\int_F q_F^T\DIV_F\bvec{w}_F
  + \sum_{E\in\EF}\omega_{FE}\int_E\trE\underline{q}_E~(\bvec{w}_F\cdot\normal_{FE}),
  \]
  which shows, since $\DIV_F:\cRoly{k}(F)\to\Poly{k-1}(F)$ is an isomorphism, that $q_F^T$ only depends on $\bvec{z}_{\cvec{R},F}^\compl$ and $\trE\underline{q}_E$, quantities that are, in turn, independent of $T$.
  We therefore conclude that $q_F^T = q_F$ for all $F\in\Fh$ and all $T\in\Th$ having $F$ as a face.
  Setting $\underline{q}_h \coloneq ((q_T^T)_{T\in\Th},(q_F)_{F\in\Fh},(q_E)_{E\in\Eh},(0)_{V\in\Vh})\in\Xgrad{k}{h}$, we then have $\underline{q}_T=\underline{q}_T^T$ for all $T\in\Th$; recalling \eqref{eq:P.complex.ER:curl:T}, this concludes the proof of \eqref{eq:left-inverse.R:curl}.
  \medskip\\
  \emph{(iii)} \emph{Proof of \eqref{eq:left-inverse.R:div}.}
  Let $\uvec{w}_h\in\Xdiv{k}{h}$ be such that $\Dh{k}\uvec{w}_h = 0$.
  Since the reductions and extensions are cochain maps, as in Point (ii) above we have $\Dh{k}\Ediv{h}\Rdiv{h}\uvec{w}_h=i\lproj{0}{h}\Dh{k}\uvec{w}_h=0$.
  Hence, $\Ediv{h}\Rdiv{h}\uvec{w}_h-\uvec{w}_h\in\Ker\Dh{k}$, and the exactness of the local DDR complex yields, for all $T\in\Th$, the existence of $\uvec{v}_T^T\in\Xcurl{k}{T}$ such that
  \begin{equation}\label{eq:P.complex.ER:div:T}
    \uvec{z}_T\coloneq\Ediv{T}\Rdiv{T}\uvec{w}_T - \uvec{w}_T
    = \uCT{k}\uvec{v}_T^T.
  \end{equation}
  The above relation implies, accounting for the definitions of the reduction and extension operators, $\lproj{0}{F}\CF{k}\uvec{v}_T^T = 0$ for all $F\in\FT$.
  By virtue of Proposition \ref{eq:pi0.CF=0} below, we can assume that $\int_E v_E^T = 0$ for all $E\in\ET$ without loss of generality.
  We additionally notice that, by the complex property of the local DDR sequence, \eqref{eq:P.complex.ER:div:T} holds up to the substitution $\uvec{v}_T^T\gets\uvec{v}_T^T + \uGT{k}\underline{q}_T^T$ with $\underline{q}_T^T\in\Xgrad{k}{T}$.
  
  We leverage these observations as described hereafter.
  Let $E\in\Eh$, denote by $\elements{E}\subset\Th$ the set of mesh elements sharing $E$, and fix one $T_E\in\elements{E}$.
  For all $T\in\elements{E}\setminus\{T_E\}$, we select $\underline{q}_E^{T}\in\Xgrad{k}{E}$ such that $q_V^{T} = 0$ for all $V\in\VE$ and $v_E^{T} + \GE{k}\underline{q}_E^{T} = v_E^{T_E}\eqcolon v_E$ (the existence of such $\underline{q}_E^{T}$ is guaranteed by the conditions $\int_E (v_E-v_E^{T}) = 0$).
  
  Given a mesh element $T\in\Th$, we use the vectors $\underline{q}_E^T = (q_E^T, (0)_{V\in\VE})$, $E\in\ET$, constructed above to form a vector $(0, (q_F^T)_{F\in\FT}, (q_E^T)_{E\in\ET}, (0)_{V\in\VT})\in\Xgrad{k}{T}$ with face components selected so as to ensure that $\hat{\uvec{v}}_T^T\coloneq\uvec{v}_T^T + \uGT{k}\underline{q}_T^T$ can be glued together at faces shared by two different elements. Let us describe this selection.
  By construction, the edge components of $\hat{\uvec{v}}_T^T$ are independent of $T$, hence we denote them without the superscript ``$T$''.
  We moreover notice that, by \eqref{eq:P.complex.ER:div:T} combined with the definition \eqref{eq:CF} of the face curl, for all $F\in\FT$, $\hat{\bvec{v}}_{\cvec{R},F}^T = \hat{\bvec{v}}_{\cvec{R},F}$ with $\hat{\bvec{v}}_{\cvec{R},F}\in\Roly{k-1}(F)$ such that
  \[
  \int_F\hat{\bvec{v}}_{\cvec{R},F}\cdot\VROT r_F
  = \int_F z_F~r_F + \sum_{E\in\EF}\omega_{FE}\int_Ev_E~r_F
  \qquad\forall r_F\in\Poly{k}(F).
  \]
  The right-hand side of the above expression does not depend on $T$, showing that, as announced, $\hat{\bvec{v}}_{\cvec{R},F}$ is indeed single-valued (that is, it only depends on $F$ and not the elements to which $F$ belongs).
  For any $F\in\FT$ shared with an element $T'\in\Th$, we then proceed as follows to select $q_F$ in order to ensure that the face components in $\cRoly{k}(F)$ are also single-valued:
  If $\omega_{TF} = 1$, we let $q_F^T = 0$, otherwise we take $q_F^T\in\Poly{k-1}(F)$ such that the components of $\hat{\uvec{v}}_T^T = \uvec{v}_T^T + \uGT{k}\underline{q}_T^T$ and $\hat{\uvec{v}}_{T'}^{T'}=\uvec{v}_{T'}^{T'} + \uGTp{k}\underline{q}_{T'}^{T'}$ (with $\underline{q}_{T'}^{T'} = (0,(0)_{F\in\faces{T'}},(q_E^{T'})_{E\in\edges{T'}},(0)_{V\in\vertices{T'}})$) on $\cRoly{k}(F)$ match, i.e., recalling the definition \eqref{eq:cGF} of $\cGF{k}$,
  \[
  \int_F q_F^T~\DIV\bvec{y}_F
  = \sum_{E\in\EF}\omega_{FE}\int_E\trE(\underline{q}_E^T-\underline{q}_E^{T'})~(\bvec{y}_F\cdot\normal_F)
  + \int_F(\bvec{v}_{\cvec{R},F}^{\compl,T} - \bvec{v}_{\cvec{R},F}^{\compl,T'})\cdot\bvec{y}_F
  \qquad\forall\bvec{y}_F\in\cRoly{k}(F).
  \]
  This relation defines $q_F^T$ uniquely since $\DIV:\cRoly{k}(F)\to\Poly{k-1}(F)$ is an isomorphism.
  This concludes the construction of local vectors $\hat{\uvec{v}}_T^T$, $T\in\Th$, that can be glued at faces to form a global vector $\uvec{v}_h\in\Xcurl{k}{h}$ satisfying \eqref{eq:left-inverse.R:div}.
  \medskip\\
  \emph{(iv)} \emph{Proof of \eqref{eq:left-inverse.R:Pk}.}
  Let $r_h\in\Poly{k}(\Th)$ and define $\uvec{w}_h=((\bvec{w}_{\cvec{G},T},\bvec{0})_{T\in\Th},(0)_{F\in\Fh})\in\Xdiv{k}{h}$ such that, for all $T\in\Th$, $\bvec{w}_{\cvec{G},T}\in\Goly{k-1}(T)$ satisfies
  \begin{equation}\label{eq:C2.tail.def.w}
    \int_T \bvec{w}_{\cvec{G},T}\cdot\GRAD s_T = -\int_T (\lproj{0}{T}r_T-r_T)s_T\quad\forall s_T\in\Poly{k}(T).
  \end{equation}
  Since $\GRAD:\Poly{0,k}(T)\to\Goly{k-1}(T)$ is an isomorphism, this formula restricted to $s_T\in\Poly{0,k}(T)$ entirely defines $\bvec{w}_{\cvec{G},T}$; we then notice that \eqref{eq:C2.tail.def.w} is also satisfied for constant polynomials $s_T$ (since $\int_T (\lproj{0}{T}r_T-r_T)=0$), which justifies that it holds for all $s_T\in\Poly{k}(T)$.
  
  Since the components of $\uvec{w}_h$ on the faces are all equal to zero, \eqref{eq:C2.tail.def.w} and the definition \eqref{eq:DT} of $\DT{k}$ shows that $\DT{k}\uvec{w}_T=\lproj{0}{T}r_T-r_T$ for all $T\in\Th$, which shows that $\Dh{k}\uvec{w}_h=\lproj{0}{h}r_h-r_h$ and concludes the proof of \eqref{eq:left-inverse.R:Pk}.
\end{proof}

\begin{proposition}[Elements of the local curl space with zero-average face curl]\label{eq:pi0.CF=0}
  Let $T\in\Th$ and $\uvec{v}_T\in\Xcurl{k}{T}$ be such that, for all $F\in\FT$ with area $|F|$,
  \begin{equation*}%\label{eq:CF=0}
    \lproj{0}{F}\CF{k}\uvec{v}_F
    = -\frac{1}{|F|}\sum_{E\in\EF}\omega_{FE}\int_E v_E
    = 0.
  \end{equation*}
  Then, there exists $\uvec{w}_T\in\Xcurl{k}{T}$ such that 
  \begin{equation}\label{eq:uCT.w=uCT.v}
    \lproj{0}{E}w_E = 0\ \text{ for all $E\in\ET$},\quad\text{and}\quad
    \uCT{k}\uvec{w}_T = \uCT{k}\uvec{v}_T.
  \end{equation}
\end{proposition}

\begin{proof}
  Setting, for all $E\in\ET$, $w_E\coloneq v_E - \lproj{0}{E} v_E$ ensures that the first condition in \eqref{eq:uCT.w=uCT.v} is verified.
  Recalling the definition \eqref{eq:CF} of the face curl, and noticing that we already have $\lproj{0}{F}\CF{k}\uvec{w}_F=0$, enforcing $\CF{k}\uvec{w}_F = \CF{k}\uvec{v}_F$ amounts to selecting $\bvec{w}_{\cvec{R},F}\in\Roly{k-1}(F)$ such that
  \[
  \int_F\bvec{w}_{\cvec{R},F}\cdot\VROT_Fr_F
  = 
  \int_F\bvec{v}_{\cvec{R},F}\cdot\VROT_Fr_F
  -\sum_{E\in\EF}\omega_{FE}\int_E\lproj{0}{E}v_E~r_F
  \qquad
  \forall r_F\in\Poly{0,k}(F).
  \]
  Since $\VROT_F:\Poly{0,k}(F)\to\Roly{k-1}(F)$ is an isomorphism, this condition defines, for all $F\in\FT$, a unique value for $\bvec{w}_{\cvec{R},F}$.
  We then set $\uvec{w}_F \coloneq (\bvec{w}_{\cvec{R},F},\bvec{0},(w_E)_{E\in\EF})$.
  The equality of face curls enforced above implies, by the relation between face and element curls of \cite[Proposition 4]{Di-Pietro.Droniou:21*2}, that $\Gproj{k-1}{T}\cCT{k}\uvec{v}_T = \Gproj{k-1}{T}\cCT{k}\uvec{w}_T$.
  Finally, to enforce $\Gcproj{k}{T}\cCT{k}\uvec{v}_T = \Gcproj{k}{T}\cCT{k}\uvec{w}_T$, recalling the definition \eqref{eq:cCT} of the element curl, we select $\bvec{w}_{\cvec{R},T}\in\Roly{k-1}(T)$ such that, for all $\bvec{z}_T\in\cGoly{k}(T)$,
  \[
  \int_T\bvec{w}_{\cvec{R},T}\cdot\CURL\bvec{z}_T
  = \int_T\bvec{v}_{\cvec{R},T}\cdot\CURL\bvec{z}_T
  + \sum_{F\in\FT}\omega_{TF}\int_F\trFt{k}(\uvec{v}_F-\uvec{w}_F)\cdot(\bvec{z}_T\times\normal_F).
  \]
  This relation defines $\bvec{w}_{\cvec{R},T}$ uniquely since $\CURL:\cGoly{k}(T)\to\Roly{k-1}(T)$ is an isomorphism.
  The vector $\uvec{w}_T=(\bvec{w}_{\cvec{R},T},\bvec{0},(\bvec{w}_{\cvec{R},F},\bvec{0})_{F\in\FT},(w_E)_{E\in\ET})$ constructed above then fulfils the second condition in \eqref{eq:uCT.w=uCT.v}, thus concluding the proof.
\end{proof}

\begin{remark}[Chain homotopy between the $\DDR{k}$ and $\DDR{0}$ complexes]
The proof of Theorem \ref{thm:ddr:cohomology}, can also be interpreted through the concept of \emph{chain homotopy}.
Specifically, it can be shown that the reduction $\Rbullet{h}$ is a chain equivalence with extension $\Ebullet{h}$ as a chain-homotopy inverse for $\bullet\in \{\GRAD, \CURL, \DIV \}$. Since \eqref{eq:left-inverse.E} already shows that $\Rbullet{h} \Ebullet{h} = \Id_{\Xbullet{k}{h}}$, it is sufficient to find a chain homotopy between $\Ebullet{h} \Rbullet{h}$ and $\Id_{\Xbullet{k}{h}}$, namely, mappings $\Dgrad:\Xgrad{k}{h} \to \mathbb R$, $\Dcurl : \Xcurl{k}{h} \to \Xgrad{k}{h}$, $\Ddiv: \Xdiv{k}{h} \to \Xcurl{k}{h}$, $\mathcal D_{\Poly{k}(\Th)}: \Poly{k}(\Th) \to \Xdiv{k}{h}$ such that
\begin{subequations}\label{eq:left-inverse.D}
  \begin{gather}\label{eq:left-inverse.D:grad}
    \text{%         
      $\Id_{\Xgrad{k}{h}} - \Egrad{h}\Rgrad{h}  = \Igrad{k}{h} \Dgrad + \Dcurl \uGh{k}$, 
    }
    \\\label{eq:left-inverse.D:curl}
    \text{%
      $\Id_{\Xcurl{k}{h}} - \Ecurl{h}\Rcurl{h}  = \uGh{k}\Dcurl + \Ddiv \uCh{k}$,
    }
    \\\label{eq:left-inverse.D:div}
    \text{%
      $\Id_{\Xdiv{k}{h}} - \Ediv{h}\Rdiv{h}  = \uCh{k}\Ddiv + \mathcal D_{\Poly{k}(\Th)} \Dh{k}$,
    }
    \\\label{eq:left-inverse.D:pol}
    \text{%
      $\Id_{\Poly{k}(\Th)} - \lproj{0}{h}  = \Dh{k} \mathcal D_{\Poly{k}(\Th)}$
    }.
  \end{gather}
\end{subequations}
The design of these mappings relies on two key points. First, the fields $\underline{q}_h\in\Xgrad{k}{h}$ and $\uvec{v}_h\in\Xcurl{k}{h}$ constructed in Points  (ii) and (iii) of the proof of Theorem \ref{thm:ddr:cohomology}, respectively, are unique if we impose $\Rgrad{h}\underline{q}_h=\underline{0}$ and $\Rcurl{h}\uvec{v}_h=\uvec{0}$. Second, by setting $\Pibullet{h} \coloneq \Id_{\Xbullet{k}{h}} - \Ebullet{h}\Rbullet{h}$ for $\bullet\in\{\GRAD,\CURL,\DIV\}$, 
it can be shown that $\Pibullet{h}$ is a cochain map and that $\Image\Pibullet{h} = \XbulletZR{k}{h}$.
\end{remark}

\begin{remark}[Zero-reduction sub-complex]
Let us define the zero-reduction subspaces of the DDR spaces \eqref{eq:ddr.spaces} (which are simply the kernels of the reductions):
\[
\text{%
  $\XbulletZR{k}{h} \coloneq\left\{
  \underline{x}_h\in\Xbullet{k}{h}\st\Rbullet{h}\underline{x}_h = \underline{0}
  \right\}$
  for $\bullet\in\{\GRAD,\CURL,\DIV\}$, and
  $\widetilde{P}_h^k\coloneq\left\{r_h\in\Poly{k}(\Th)\st\lproj{0}{h}r_h = 0\right\}$.
}
\]
It can then be checked that the zero-reduction subcomplex
\begin{equation}\label{eq:ddrZR.complex}
  \begin{tikzcd}
    0\arrow{r}{\Igrad{k}{h}} &[1.5em] \XgradZR{k}{h}
    \arrow{r}{\uGh{k}} & \XcurlZR{k}{h}
    \arrow{r}{\uCh{k}} & \XdivZR{k}{h} 
    \arrow{r}{\Dh{k}}  & \widetilde{P}_h^k
    \arrow{r} & 0,
  \end{tikzcd}
\end{equation}
is well defined, and an equivalent formulation of Lemma \ref{lem:propC2} is that this subcomplex is exact, irrespective of the topology of $\Omega$.
With a standard inclusion of the $\DDR{0}$ spaces into the $\DDR{k}$ spaces, we have $\Xbullet{k}{h}=\XbulletZR{k}{h}\oplus \Xbullet{0}{h}$ for $\bullet\in\{\GRAD,\CURL,\DIV\}$ and $\Poly{k}(\Th)=\widetilde{P}_h^k\oplus \Poly{0}(\Th)$.
Hence, the exactness of the zero-reduction subcomplex is another way of seeing that the information on the topology of the domain is completely encapsulated in the lowest-order portion of the $\DDR{k}$ complex.
\end{remark}

\begin{remark}[Cohomology of FEEC]
The standard technique to analyse the cohomology of FEEC is through the usage of $L^2$-bounded projection operators between the de Rham complex \eqref{eq:de-rham} and the FEEC complex \cite[Section 7.5]{Arnold:18}. The design of these operators is not trivial \cite[Section 5]{Arnold.Falk.ea:06}, and they must satisfy approximation properties which sometimes requires to consider fine enough meshes.

Our approach, on the contrary, does not rely on any analytical property of a cochain map between the discrete and continuous complexes. It only uses the fact (widely known in algebraic topology), that CW complexes have the same cohomology as the continuous de Rham complex, and establishes that the higher-order portion of the discrete complex does not play any role in its cohomology. 

In passing, we note that this approach could be applied to the FEEC setting, using as reduction maps the canonical embedding of a low-order polynomial space into its high-order version, and as extension maps the interpolations based on the degrees of freedom. We also refer the reader to the recent work \cite{Bonaldi.Di-Pietro.ea:23}, in which two exterior calculus polytopal complexes are designed and their cohomology is analysed using the same technique as above.
\end{remark}

\begin{remark}[Computation of the cohomology spaces of the DDR complex]\label{rem:generators}
The proof of Theorem \ref{thm:ddr:cohomology} suggests the following efficient procedure to construct generators of cohomology spaces $\coH{1, (k)}$ and $\coH{2, (k)}$ (the relevant cohomology spaces in practical applications, since $\coH{0,(k)}$ and $\coH{3,(k)}$ are trivial). 
Let us focus on $\coH{1, (k)}$, the other case being similar. Efficient graph-based algorithms are available to compute generators
of the first cohomology space of the CW complex, see for instance \cite{Dotko.Specogna:13}. Let $(\mathfrak{g}_j)_{j\in \{1, \dots, b_1\}}$ be representatives in $\Eh^*$ of these generators.
An inverse $\kappa_{\CURL}^{-1}$ of the de Rham map in \eqref{eq:DDR0.CW} can trivially be constructed, and provides representatives $(\kappa_{\CURL}^{-1}\mathfrak{g}_j)_{j\in \{1, \dots, b_1\}}$ in $\Xcurl{0}{h}$ of generators of $\coH{1,(0)}$. We can then explicitly get $\Ecurl{h}$ through local computations on mesh elements (see \eqref{eq:E.curl.h}), and thus obtain representatives $(\Ecurl{h}\kappa_{\CURL}^{-1}\mathfrak{g}_j)_{j\in \{1, \dots, b_1\}}$ in $\Xcurl{k}{h}$ of generators of $\coH{1,(k)}$.
\end{remark}

\section{Conclusion and perspectives}\label{sec:conclusion}
In this work, we establish that the discrete de Rham complex DDR($k$) and its serendipity version are isomorphic in cohomology to the continuous de Rham complex.
The proof hinges on the construction of reduction and extension cochain maps between the $\DDR{k}$ complex for $k\ge 1$ and the $\DDR{0}$ complex, which is representative of the usual low-order cochain complex defined on the CW complex of the mesh.
This result represents an essential theoretical and practical step towards using the DDR construction (and, more generally, high-order polytopal complexes) to discretise physical problems on domains with non-trivial topologies.

On the theoretical side, one of the virtues of our result is that extension maps allow to extend standard cohomology constructions of the low-order cochain complex to the high-order complex.
As an example, the well-posedness of certain electromagnetic boundary value problems requires the usage of so-called \emph{relative cohomology spaces}.
In this case, instead of trying to develop a relative cohomology theory on the sequence $\DDR{k}$, we can exploit extension maps to define such relative cohomology spaces starting from those of $\DDR{0}$, where a standard de Rham isomorphism (relative, in this case) with respect to the continuous de Rham complex can be readily established. 

On the practical side, the computation of cohomology spaces consists in finding the quotient vector spaces \eqref{eq:cohomology.spaces:discrete}.
Obtaining bases of these quotient spaces requires the solution of expensive linear algebra problems.
The fundamental computational advantage of our construction is that it provides an explicit and inexpensive way to find generators of the cohomology spaces of the DDR complex starting from those of the CW complex associated with the mesh. 

Future work will explore the application of this result to obtain representations of the cohomology spaces of the DDR complex, and the study of their analytical properties required for their usage in schemes for relevant problems on non-trivial topologies.

%------------------------------------------------------------------------------%

\section*{Declarations}

The authors acknowledge the partial support of \emph{Agence Nationale de la Recherche} grant ANR-20-MRS2-0004 NEMESIS.
Daniele Di Pietro also acknowledges the partial support of I-Site MUSE grant ANR-16-IDEX-0006 RHAMNUS.
J\'er\^ome Droniou was partially supported by the Australian Research Council through the Discovery Projects funding scheme (Project No.\ DP210103092).

The authors have no relevant financial or non-financial interests to disclose.

%------------------------------------------------------------------------------%
% Bibliography
%------------------------------------------------------------------------------%

\printbibliography

\end{document}